\documentclass[reqno,12pt]{amsart}
\usepackage{dsfont}
\usepackage{amssymb,psfrag,graphicx}
\usepackage{epstopdf} 

\usepackage[colorlinks=true]{hyperref}
\hypersetup{urlcolor=blue, citecolor=red}

\usepackage[OT1]{fontenc}
\usepackage[latin1]{inputenc}

\newcommand{\be}{\begin{equation}}
\newcommand{\ee}{\end{equation}}
\newcommand{\bea}{\begin{eqnarray}}
\newcommand{\eea}{\end{eqnarray}}
\newcommand{\bu}{\mathbf u}
\newcommand{\bv}{\mathbf v}
\newcommand{\bw}{\mathbf w}
\newcommand{\bx}{\mathbf x}
\newcommand{\f}{\mathbf f}
\newcommand{\rw}{\textrm{\rm w}}
\newcommand{\rb}{\textrm{\rm b}}
\newcommand{\rc}{\textrm{\rm c}}
\newcommand{\rloc}{\textrm{\rm loc}}

\newcommand{\KU}{\Pi_{t_0}^{-1}K\cap\mathcal{U}}
\newcommand{\KUI}{\Pi_{t_0}^{-1}K\cap\mathcal{U}_I}

\newcommand{\mM}{\mathcal M}
\newcommand{\mX}{\mathcal X}
\newcommand{\mA}{\mathcal A}

\newcommand{\mC}{\mathcal C}
\newcommand{\mK}{\mathcal K}

\newcommand{\mP}{\mathcal P}
\newcommand{\mS}{\mathcal S}
\newcommand{\mY}{\mathcal Y}
\newcommand{\mU}{\mathcal U}
\newcommand{\mV}{\mathcal V}

\newcommand{\mZ}{\mathcal Z}

\newcommand{\F}{\mathbf F}

\newcommand{\fB}{\mathfrak{B}}
\newcommand{\fK}{\mathfrak{K}}
\newcommand{\fF}{\mathfrak{F}}
\newcommand{\fL}{\mathfrak{L}}

\newcommand{\wsconv}{\stackrel{w*}{\rightharpoonup}}
\newcommand{\wconv}{\stackrel{wsc*}{\rightharpoonup}}
\newcommand{\rd}{{\text{\rm d}}}
\newcommand{\Cloc}{\mathcal{C}_{\rloc}}
\newcommand{\MX}{\mathcal{M}(X)}
\newcommand{\PX}{\mathcal{P}(X)}
\newcommand{\MXt}{\mathcal{M}(X,\text{tight})}
\newcommand{\PXt}{\mathcal{P}(X,\text{tight})}

\theoremstyle{plain}
\newtheorem{thm}{Theorem}[section]
\newtheorem{lem}{Lemma}[section]
\newtheorem{prop}{Proposition}[section]
\newtheorem{cor}{Corollary}[section]
\newtheorem{defs}{Definition}[section]
\theoremstyle{definition}
\newtheorem{rmk}{Remark}[section]
\newtheorem{prob}{Problem}[section]

\newcommand{\comment}[1]{\\\indent\textcolor{red}{\framebox{\parbox{0.9\textwidth}{\textbf{#1}}}}\\}

\renewcommand{\comment}[1]{}

\begin{document}

\title{Abstract Framework for the Theory of Statistical Solutions}
\author{Anne C. Bronzi}
\author{Cecilia F. Mondaini}
\author{Ricardo M. S. Rosa}
\address[A. C. Bronzi]{Instituto de Matem\'atica, Estat\'{\i}stica e Computa\c{c}\~ao Cient\'{\i}fica, Universidade Estadual de Campinas, Campinas, São Paulo, 13083-859, Brazil}
\address[C. F. Mondaini]{Instituto de Matem\'atica, Universidade Federal do Rio de Janeiro, Caixa Postal 68530 Ilha do Fund\~ao, Rio de Janeiro, RJ 21941-909,
  Brazil.}
\address[R. M. S. Rosa]{Instituto de Matem\'atica, Universidade Federal do Rio de Janeiro, Caixa Postal 68530, Ilha do Fund\~ao, Rio de Janeiro, RJ 21941-909,
  Brazil.}

\email[A. C. Bronzi]{annebronzi@ime.unicamp.br}
\email[C. F. Mondaini]{cfmondaini@gmail.com}
\email[R. M. S. Rosa]{rrosa@im.ufrj.br}

\begin{abstract}
An abstract framework for the theory of statistical solutions is developed for general evolution equations, extending the theory initially developed for the three-dimensional incompressible Navier-Stokes equations. The motivation for this concept is to model the evolution of uncertainties on the initial conditions for systems which have global solutions that are not known to be unique. Both concepts of statistical solution in trajectory space and in phase space are given, and the corresponding results of existence of statistical solution for the associated initial value problems are proved. The wide applicability of the theory is illustrated with the very incompressible Navier-Stokes equations, a reaction-diffusion equation, and a nonlinear wave equation, all displaying the property of global existence of weak solutions without a known result of global uniqueness.
\end{abstract}


\maketitle

\section{Introduction}

The concept of statistical solution was introduced for the study of turbulence in incompressible Newtonian fluid flows. In a turbulent flow, most relevant physical quantities (e.g. velocity, kinetic energy, energy dissipation) display a wild variation in space and time, while displaying a more orderly behavior when averaged in space or time (see e.g. \cite{Taylor35, Batchelor1953, Hinze1975, MoninYaglom1975, Frisch1995, Lesieur1997}). This behavior appears, in fact, for different realizations of the flow, with somehow ``universal'' properties, so that one is led to consider averages with respect to an ensemble of flows, in an attempt to capture common properties of turbulent flows.

One can work with ensemble average in a formal sense, without worrying about the regularity of the solutions of the system, or one can be more strict and work with some notion of weak solution of the system for which existence results for the corresponding initial value problem are available. The statistical solution were introduced exactly with this later purpose in mind: they have been defined to model, in a rigorous way, the evolution of ensembles of weak solutions of the incompressible Navier-Stokes equations, as a foundation for a rigorous treatment of turbulent flows. 

Since its inception, the theory of statistical solutions for the incompressible Navier-Stokes equations has been the basis for a growing number of rigorous results for turbulent flows (e.g. \cite{Foias74, CFM94, Bercovicietal95, Foias97, Fursikov99, FMRT01, FJMR2002, RRT08, Rosa2009}). The concept of statistical solution has also been successfully adapted to a number of other models, particularly fluid flow models, but also other types of nonlinear partial differential equations (e.g. \cite{Capinski1979, Arsenev1979b, VishikKomech1982, Kharchenko1982, Sobolev1983, Chueshov1983, Vasileva1983, PulvirentiWick1985, Chae1991, CapinskiCutland1993, ConstantinJiahong1997, Kim1998, Panov2002, ConstantinRamos2007, Kelliher2009,RamosTiti2010,BronziRosa2014}). 

In fact, the notion of ensemble average is relevant for any evolution system displaying a complicated dynamics, in which uncertainties in the initial condition are of crucial concern. The concept of ensemble averages is directly related to the evolution of a probability distribution of initial conditions. In a well-posed system, with a well-defined semigroup $\{S(t)\}_{t\geq 0}$, the evolution $\{\mu_t\}_{t\geq 0}$ of the probability distribution of the state of the system at each time $t$ is just the transport, or push-forward, $\mu_t = S(t)\mu_0$, of the initial probability measure $\mu_0$, by the semigroup (more precisely, $\mu_t(E) = \mu_0(S(t)^{-1}E)$, for any Borel subset $E$ of the phase space). The difficulty is to extend this definition to obtain the distributions $\mu_t$ for systems in which $\{S(t)\}_{t\geq 0}$ might not be defined. This was the aim of the concept of statistical solution in the particular and fundamental case of the three-dimensional incompressible Navier-Stokes equations.

Two main definitions of statistical solutions have been introduced in the 1970's. First, Foias (\cite{Foias72, Foias73}), in works stemmed from discussions with Prodi (see e.g. \cite{FoiasProdi76}), introduced the concept of statistical solution in phase space, consisting of a family of measures on the phase space of the Navier-Stokes system, parametrized by the time variable, and representing the evolution of the probability distribution of the state of the system. Then, Vishik and Fursikov (\cite{VF78, VF88}) introduced the notion of a space-time statistical solution, which is that of a single measure defined on the space of trajectories of the system, hence encompassing both space and time variables at the same time.

Still in the mid to late 1970's, it is worth mentioning the work by Ladyzhenskaya and Vershik (\cite{LadyVershik1977}), presenting a different proof of existence of statistical solution (in phase space) using a representation theorem by Castaing (\cite{Castaing1969}), for the measurability of multivalued solution maps, to overcome the fact that the solution operator is not a continuous, single-valued map. There is also the work by Arsen'ev (\cite{Arsenev1976}) using a measurable selection argument to construct a statistical solution and, in fact, already introducing a notion of space-time statistical solution, which at first had some restrictions on the initial measure, but that were eventually relaxed (\cite{Arsenev1979}). Another proof of existence of space-time statistical solutions, based on nonstandard analysis, was given in the early 1990's, by Capinski and Cutland (see e.g. \cite{CapinskiCutland1992}). 

More recently, inspired by the definition given in \cite{VF78, VF88} and using a different approach for the construction of statistical solutions, by approximating the initial measure by convex combinations of Dirac delta functions, given in \cite{FMRT}, Foias, Rosa and Temam \cite{FRT2010} considered the idea of space-time statistical solution of Vishik and Fursikov but with slightly different hypotheses that make it more amenable to analysis. Projecting that modified space-time statistical solution to the phase space at each time yields a family of measures which is a particular type of statistical solution in phase space, hence bridging the two notions given earlier. 

This work \cite{FRT2010} inspired us to look for a more general abstract framework for the definition of statistical solutions for various types of systems and with reasonable conditions guaranteeing the existence of statistical solutions for the associated initial-value problem. The constructions developed in the previous works were very much based on the structure of the equations, but nevertheless we hoped to find a framework general enough to be applied to different types of systems. 

The first result was published in \cite{bmr2014} and was sufficiently general to apply naturally to an ample class of problems. However, there is one condition in \cite{bmr2014} which has a more involved formulation and made us feel that there should be a simpler way to attack the problem. Moreover, \cite{bmr2014} only addressed the notion of space-time statistical solution, which we call here a trajectory statistical solution or a statistical solution in trajectory space.

In our current work, we keep the same framework of \cite{bmr2014} for the trajectory statistical solutions but with a simpler set of hypotheses that avoid the more involved condition in our previous work. The proof is also simpler with this new set of hypotheses. Furthermore, we address not only statistical solutions in trajectory space but also in phase space, presenting a general setting for the definition of those two types of solutions and the corresponding theorems of existence for the associated initial value problems. The proof also significantly simplifies the known proofs for concrete examples of equations, particularly the Navier-Stokes equations.

We should mention that the concept of statistical solution had been previously extended to a few frameworks encompassing some particular classes of differential equations (e.g. \cite{IllnerWick1981, Golec1993, Barbablah1995, Haragus1999, Slawik2008}), but none of them nearly as general as done here.

The hypotheses needed for our existence results are natural and not difficult to verify, relying essentially on properties of the set of individual solutions of the system. In fact, \emph{the key ingredient in our approach is to look for topological and measure-theoretic properties of the set of solutions instead of looking at the structure of the equation.}

The abstract framework that we construct starts with a Hausdorff space $X$, an interval $I\subset\mathds{R}$, the space of continuous paths $\mX=\Cloc(I,X)$ endowed with the compact-open topology, and a subset $\mU$ of $\mX$. There is no system of equations or solution operator at this abstract level, allowing for a wide range of applications, even for evolution problems not arising directly from differential equations. Only in the applications is that those objects will be realized, with the space $X$ being interpreted as the phase space of the system; the interval $I$, as a time interval for the evolution system; the space $\mX$, as a space-time function space in which the solutions, or trajectories, of the system are included; and $\mU$, as a subset of solutions, or trajectories, of the system. In the case of partial differential equations, the set $\mU$ is usually the whole set of weak solutions of the problem or some particular subset of solutions (e.g. Leray-Hopf weak solutions or suitable weak solutions for the three-dimensional Navier-Stokes equations, viscosity solutions for evolutionary Hamilton-Jacobi equations).

In this abstract setting, a $\mU$-trajectory statistical solution or simply a trajectory statistical solution is a tight Borel probability measure $\rho$ carried by a Borel subset of $\mX$ which is included in the set $\mU$ (see Definition \ref{def-stat-sol}), i.e. 
\begin{equation}
  \mbox{there exists a Borel subset } \mV \subset \mU \mbox{ such that } \rho(\mX\setminus \mV) = 0.
\end{equation}
We are tempted to say that $\rho$ is carried by $\mU$, but at this point it is not assumed that $\mU$ is a Borel subset of $\mX$. This allows the framework to be applied to systems for which it is not known whether the subset of solutions is a Borel set in the appropriate space. Nevertheless, $\mU$ is certainly measurable with respect to the Lebesgue completion $\bar\rho$ of $\rho$, and $\bar\rho(\mU) = 1$ (see Remark \ref{urhobarmeasurable}).

The terminology of trajectory statistical solution is inspired by the notion of trajectory attractor given, for instance, in \cite{ChepVishik} (see also \cite{sell1996}), and the connection between the two will become more apparent in a future work, when we focus on the case of a stationary trajectory statistical solution.

For the initial value problem, we consider the interval $I$ as being closed and bounded on the left, with left endpoint $t_0$, and consider a tight Borel probability measure $\mu_0$, defined on the space $X$, as the initial probability distribution of the state of the system. We also consider the projection operator $\Pi_{t_0}:\mX\rightarrow X$ which takes $u$, in $\mX$, into its value $\Pi_{t_0}u = u(t_0)$, at time $t_0$. Then, the initial value problem for trajectory statistical solutions is simply to find a $\mU$-trajectory statistical solution $\rho$ such that $\Pi_{t_0}\rho =\mu_0$ (see Problem \ref{ivp_tss}).

The existence of trajectory statistical solutions for this initial value problem relies on essentially three conditions: 
(i) a set-theoretic surjective condition on the projection of $\mU$ at the initial time, which in the applications is simply a statement of existence of ``global'' solutions, over the whole interval $I$, for every initial condition in the ``phase'' space $X$; (ii) a topological compactness condition on the subset $\mU$ associated with a certain family of compact sets of initial conditions, which in applications follows from appropriate compact embedding theorems typically used for the existence of individual solutions of the system; and (iii) a joint topological and measure-theoretic property saying that the family of compact sets of initial conditions in the previous condition is sufficiently large to approximate from below any tight Borel probability measure on $X$ (see Theorem \ref{existencestatsol}).

It might happen that the set of initial conditions for the existence of individual solutions does not coincide with the space in which the continuity in time holds. In this case, the result can in fact be modified to yield statistical solutions only carried by the ``good'' set of initial conditions (see Theorems \ref{existencestatsolweaker} and  \ref{thmextensionhdoubleprimephasespace} and Remark \ref{extensionhdoubleprime}).

Next, we turn to the concept of statistical solution in phase space. At this point, we need to be more specific about the evolution problem. Hence, we look for statistical solutions of an evolution equation of the form
\begin{equation}
  \label{diffeq}
  u_t = F(t,u). 
\end{equation}
For this differential equation to make sense, we need a vector space structure. With that in mind, we start with the same setting as above and add another Hausdorff space $Z$ and a topological vector space $Y$, such that $Z\subset X \subset Y_{\rw^*}'$, with continuous injections, and where $Y_{\rw^*}'$ is the dual of $Y$ endowed with the weak star topology. We denote the duality product between $Y$ and $Y'$ by $\langle \cdot, \cdot \rangle_{Y',Y}$. We consider a function $F: I \times Z \rightarrow Y'$ and assume that $u\in \mU$ is such that $u(t)$ belongs to $Z$ for almost every $t\in I$. With this framework, the definition of statistical solution in phase space is that of a family of measures $\{\mu_t\}_{t\in I}$ which satisties \eqref{diffeq} in a suitable weak-star sense in the mean, i.e.
\begin{equation}
  \label{diffeqmean}
  \frac{\rd}{\rd t} \int_X \Phi(u)\;\rd\mu_t(u) = \int_X \langle F(t,u),\Phi'(u) \rangle_{Y',Y} \;\rd\mu_t(u),
\end{equation}
in the distribution sense on $I$, for appropriate cylindrical test functions $\Phi:Y'\rightarrow \mathds{R}$, in conjunction with some measurability and integrability properties (see Definition \ref{def-statsolphasespace}). If a statistical solution in phase space $\{\rho_t\}_{t\in I}$ is obtained as the family of projections $\rho_t=\Pi_t\rho$, $t\in I$, of a trajectory statistical solution $\rho$, we say that it is a projected statistical solution.

For the initial value problem for statistical solutions in the phase space, we first make the same assumptions as above for the subset $\mU$, hence obtaining a $\mU$-trajectory statistical solution starting with a given initial measure $\mu_0$. Then, we assume that, for every $u\in \mU$, the function $\langle u(t),v\rangle_{Y',Y}$ is absolutely continuous on $I$, for every $v\in Y$, with \eqref{diffeq} being satisfied in a weak sense, namely that
\begin{equation}
  \label{diffeqweak}
  \frac{\rd }{\rd t} \langle u(t), v\rangle_{Y',Y} = \langle F(t,u(t)), v \rangle_{Y', Y}, 
\end{equation}
in the distribution sense on $I$, for every $v\in Y$. We add measurability conditions on the spaces and on $F$ which guarantee the appropriate measurability of the Nemytskii-type operator $(t,u) \mapsto F(t,u(t))$. One measurability condition is that every Borel subset of $Z$ be also a Borel subset of $X$, which is a condition satisfied in all the examples that we are aware of (see Section \ref{subsecinjectborel}). We also add integrability conditions on the right hand side of \eqref{diffeqweak}, for each trajectory in $\mU$, and an associated integrability condition on the initial measure.  Those integrability conditions are related to a~priori estimates which are natural for the system. Then, under these conditions, we obtain that the initial value problem has a statistical solution in phase space for any tight Borel probability measure satisfying the integrability condition for the initial measure (Theorem \ref{thm-existence-projectedss}). Part of the proof of this result consists in showing that, under suitable conditions, the family $\{\rho_t\}_{t\in I}$ of the projections $\rho_t = \Pi_t\rho$, $t\in I$, of a trajectory statistical solution $\rho$ is a statistical solution in phase space (see Theorem \ref{thm_projectedss_is_ss}).

The result on the existence of a statistical solution for the associated initial value problem is complemented with a result saying essentially that any energy-type inequality (or equality) valid for the individual solutions holds also, in average, for the trajectory statistical solution and, hence, also for the projected statistical solutions (Propositions \ref{prop-energy-ineq} and \ref{prop-energy-eq}).

The proof of existence of a trajectory statistical solution for the initial value problem is based on the Krein-Milmam approximation of the initial measure by convex combinations of Dirac deltas, as done in \cite{FMRT, FRT2013}. We construct a net of trajectory statistical solutions with the measure at the initial time approximating the given initial measure. At the limit, we obtain the desired trajectory statistical solution. In order to pass to the limit, one needs a compactness result for measures. A suitable result of this kind in our abstract framework is the one developed by Topsoe along his work on a generalization of Prohorov's Theorem (\cite{Prohorov56}) to spaces which are not necessarily Polish  (\cite{Topsoe,Topsoebook,Topsoe2}).

Topsoe's topology is based on semi-continuity and is finer than the weak-star topology, hence it has less compact sets, but compactness is not the main problem here. Instead, the finer topology is needed to yield more open sets, so that two different tight measures can be separated by open sets. As a consequence, the space of tight measures on $X$ under this semi-continuity-weak-star topology is a Hausdorff space, a fact that does not hold in general for the weak-star topology. If, however, the space $X$ is also completely regular, then both topologies do coincide. (See Section \ref{subsectopformeasspaces}.)

After completing the abstract theory, we present some applications of our framework. The first and natural one is the system of Navier-Stokes equations, on which our whole abstract formulation was based (Section \ref{subsecnse}). The second one is a reaction-diffusion equation (Section \ref{subsecreacdiffeq}). The third and last example is a nonlinear hyperbolic wave equation (Section \ref{subsecnonlinearwaveeq}). In each case, the system is formulated in the abstract framework and the hypotheses that yield both trajectory statistical solution and statistical solution in phase space for the corresponding initial value problems are verified. In all these examples, the spaces $Z$ and $Y$ are separable Banach spaces. In a future work, we plan to show how the theory applies to finding space-homogeneous statistical solutions for equations such as the Navier-Stokes equations on the whole space $\mathds R^3$, in which $Y$ will be taken as the inductive limit of a sequence of separable Fr\'echet spaces. 

\section{Basic Tools}

In this section, we introduce the basic concepts underlying our results. 

\subsection{Function spaces}

When working with measures on topological spaces, the topological structure is of fundamental importance. In this regard, we recall a few concepts that play an important role in this work. Besides the fundamental notion of a \textbf{Hausdorff} space, which is a topological space in which two distinct points can be separated by disjoint open sets, we recall that a topological space is \textbf{completely regular} when every nonempty closed set and every singleton disjoint from it can be separated by a continuous function. A completely regular space in which every singleton is closed is called a \textbf{Tychonoff} space. A topological space is said to be \textbf{completely metrizable} when there exists a metric compatible with the topology of the space and under which the space is complete. A topological space is called \textbf{Polish} when it is separable and completely metrizable. Separable Banach spaces and separable Fr\'echet spaces are examples of topological vector spaces which are Polish spaces.

When $X$ is a topological vector space, we denote its dual by $X'$ and the duality product is denoted by $\langle \cdot, \cdot\rangle_{X',X}$. When $X$ is endowed with its weak topology, we denote the space by $X_\rw$. Similarly, we consider $X'$ endowed with the weak-star topology, in which case we denote it by $X_{\rw^*}'$. Notice that, for any topological vector space $X$, the space $X_{\rw^*}'$ is always a Hausdorff locally convex topological vector space (\cite[Section 1.11.1]{Edwards1965}). If $X$ is a Banach space, the norm in $X$ is denoted by $\|\cdot\|_X$, while $\|\cdot\|_{X'}$ denotes the usual operator norm in the dual space.

Let $X$ be a Hausdorff space and $I\subset \mathbb R$ an arbitrary interval. Denote by $\mC(I,X)$ the \textbf{space of continuous paths} in $X$ defined on $I$, i.e. the space of all functions $u:I\rightarrow X$ which are continuous. The \textbf{compact-open topology} in $\mC(I,X)$ is the topology generated by the subbase consisting of sets of the form
\[
S(J,U) = \{ u\in \mC(I,X) \,|\, u(J)\subset U \},
\]
where $J$ is a compact subinterval of $I$ and $U$ is an open subset of $X$. When endowed with the compact-open topology, this space is denoted by $\mX = \Cloc(I,X)$ and is a Hausdorff space.

The subscript ``\emph{\rloc}'' in $\Cloc(I,X)$ refers to the fact that this topology considers compact sets in $I$. When $X$ is a uniform space, the compact-open topology in $\Cloc(I,X)$ coincides with the \textbf{topology of uniform convergence on compact subsets} (\cite[Theorem 7.11]{Kelley75}). This holds, in particular, when $X$ is a topological vector space, which is the case in the applications that are presented in Section \ref{secapplications}.

For any $t\in I$, let $\Pi_{t}:\mX\rightarrow X$ be the ``projection'' map at time $t$ defined by
\be\label{defprojectionop}
\Pi_{t}u = u(t), \quad \forall u\in\mX.
\ee
It is readily verified that $\Pi_t$ is continuous with respect to the compact-open topology.

We also consider the space of bounded and continuous real-valued functions on $X$, denoted by $\mC_\rb(X)$. When $X$ is a subset of $\mathds R^m$, $m\in \mathds N$, we also consider the space $\mC_\rc^\infty(X)$ of infinitely differentiable real-valued functions on $X$ which are compactly supported in the interior of $X$.

\subsection{Elements of measure theory}
\label{subsecelementsofmeastheory}

Let $X$ be a topological space and $\fB_X$ denote the $\sigma$-algebra of Borel sets in $X$. We denote by $\MX$ the set of finite and nonnegative Borel measures on $X$, i.e., the set of nonnegative measures $\mu$ defined on $\fB_X$ such that $\mu(X) < \infty$. The subset of $\MX$ consisting of Borel probability measures is denoted by $\PX$. The space $\MX$ can be identified with a subset of the dual space $\mC_\rb(X)'$ of the space $\mC_\rb(X)$.

A \textbf{carrier} of a measure is any measurable subset of full measure, i.e., such that its complement has null measure. If $C$ is a carrier for a measure $\mu$, we say that $\mu$ is \textbf{carried} by $C$. If a probability measure is carried by a single point $x\in X$, then it is a \textbf{Dirac measure} and it is denoted by $\delta_x$. A probability measure that can be written as a (finite) convex combination of Dirac measures is called a \textbf{discrete measure}.

Given a family of sets $\fF\subset\fB_X$, we say that a Borel measure $\mu$ on $X$ is \textbf{inner regular with respect to the family $\fF$} if
\begin{equation}
  \label{defmureg}
  \mu(A)=\sup\{\mu(F)\,|\,F\in\fF \text{ and }F\subset A\}, \qquad \forall A \in \fB_X.
\end{equation}
We say that a Borel measure $\mu$ is \textbf{outer regular with respect to the family $\fF$} if
\begin{equation}
  \label{defmuouterreg}
  \mu(A)=\inf\{\mu(F)\,|\,F\in\fF \text{ and }A \subset F\}, \qquad \forall A \in \fB_X.
\end{equation}

A \textbf{tight} measure is a nonnegative Borel measure which is inner regular with respect to the family of compact subsets of $X$ (such a measure is also called a \textbf{Radon} measure, see \cite{Bogachev}). If a finite Borel measure $\mu$ on $X$ is both tight and outer regular with respect to the family of open sets of $X$, then we say that $\mu$ is a \textbf{regular} measure. When $X$ is a Polish space, every finite Borel measure is regular (\cite[Theorem 12.7]{AB}). In case $X$ is just a metrizable space, every finite Borel measure is inner regular with respect to the family of closed subsets of $X$ and outer regular with  respect to the family of open sets of $X$ (\cite[Theorem 12.5]{AB}) (such a measure is called \textbf{normal} in \cite{AB}).

For a compact and metrizable space $X$, it follows in particular from the result in \cite[Theorem 12.5]{AB} that every finite Borel measure is tight. The metrizability is indeed a necessary condition, since it is possible to construct a finite Borel measure defined on a certain nonmetrizable compact Hausdorff space which is not tight (see \cite[Example 12.9]{AB}).

Furthermore, a net $\{\mu_\alpha\}_\alpha$ of measures in $\MX$ is said to be \textbf{uniformly tight} if for every $\varepsilon>0$ we can find a compact set $K\subset X$ such that
\[
\mu_\alpha(X \backslash K) < \varepsilon, \quad \forall \alpha.
\]

The set of measures $\mu \in \MX$ which are tight is denoted by $\MXt$. The subset of $\MXt$ consisting of probability measures is denoted by $\PXt$.

Now consider a Hausdorff space $Y$ and let $F:X\rightarrow Y$ be a Borel measurable function. Then for every measure $\mu$ on $\fB_X$ we define a measure $F\mu$ on $\fB_Y$ by
\[
F\mu(E) = \mu(F^{-1}(E)), \quad \forall E\in\fB_Y,
\]
which is called the \textbf{induced measure from $\mu$ by $F$ on $\fB_Y$}, also known as \textbf{push-forward of $\mu$ by $F$}. When $\mu$ is a tight measure and $F$ is a continuous function, the induced measure $F\mu$ is also tight.

In regard to the concept of induced measures, if $\varphi : Y \rightarrow \mathds{R}$ is a $F\mu$-integrable function then $\varphi \circ F$ is $\mu$-integrable and
\be\label{eq00}
\int_X \varphi \circ F \rd \mu = \int_Y \varphi \rd F \mu
\ee
(see \cite[Theorem 13.46]{AB}).

For the sake of notation, if $\mu \in \MX$ and $f$ is a $\mu$-integrable function, we write
\[
\mu(f) = \int_X f \rd \mu.
\]

In the case of real numbers, we are also interested in the Lebesgue measure, which we denote by $\lambda$, and in the Lebesgue subsets of intervals $I \subset \mathds R$. We denote the $\sigma$-algebra of those sets by $\fL_I$.

\subsection{Continuous injection of Borel sets}
\label{subsecinjectborel}

In the case of statistical solutions in phase space, in the abstract framework that we consider, a fundamental property that we need concerns the continuous injection of Borel subsets of a topological space into another topological space. More precisely, when we have two topological spaces $Z$ and $X$, with $Z$ continuously injected into $X$, meaning that there exists a continuous injective map $j : Z \rightarrow X$, we are interested in knowing whether the Borel subsets of $Z$ are taken into Borel subsets of $X$ by the injection $j$. Of course, this is equivalent to asking that the open subsets of $Z$ are taken into Borel subsets of $X$.

In general, this is a delicate issue. In fact, one can take $X=[0,1]$ with the usual norm inherited from $\mathds{R}$ and $Z=[0,1]$ with the zero-one norm (associated with the discrete topology), so that every subset of $Z$ is an open set, hence Borel in $Z$, but they are certainly not all included in the family of Borel subsets of $X$. It is worth noticing in this case that $X$ is a compact metric space and $Z$ is a locally compact metric space, but $Z$ is not separable. 

In the case that $Z$ and $X$ are Polish spaces, then \cite[Theorem 6.8.6]{Bogachev} guarantees that  $j(B)$ is Borel in $X$, for any Borel $B$ in $Z$. The theorem actually allows $X$ to be more generally a Souslin space, which is defined as a continuous image of a Polish space. A closely related result can be deduced from \cite[Lemmas 6.7.6 and 6.7.7]{Bourbaki1966}, allowing $X$ to be simply a metrizable space, while assuming that $Z$ is a Lusin space, which is a space intermediate between Polish and Souslin, that can be characterized as the image of a Polish space under a continuous bijective map \cite[Definition 6.4.6 and Proposition 6.4.12]{Bourbaki1966}. Thus, if $Z$ is a Lusin space and $X$ is a metrizable space, then every Borel subset of $Z$ is a Borel subset of $X$. This situation encompasses many important applications in which $Z$ and $X$ are Sobolev spaces, or other classical Banach spaces like Besov and Morrey, as long as $Z$ is separable.

Another particular situation important for us is when we consider the weak and strong topologies of a topological vector space. Since the strong topology is finer than the weak topology, every Borel set in the weak topology is also a Borel set in the strong topology. Conversely, if the topological vector space is separable and locally convex, then every strongly open set can be written as a countable union of strongly closed convex sets, and, thanks to the Hahn-Banach Theorem, every strongly closed convex set is weakly closed, so that every Borel set for the strong topology is also a Borel set for the weak topology. Therefore, in the case of separable locally convex topological vector space, both strong and weak Borel $\sigma$-algebras in fact coincide. Combining this result with that of \cite{Bourbaki1966} mentioned above, we see that if $Z$ is a Lusin space and $X$ is obtained as a metrizable and separable locally convex topological vector space endowed with its weak topology, then the Borel subsets of $Z$ are Borel subsets of $X$. This includes the case in which $X$ is a separable Banach space, or even separable Fr\'echet, endowed with its weak topology, which is the case we consider in the applications in Sections \ref{subsecnse} and \ref{subsecnonlinearwaveeq}.

The results above can be easily extended to an important situation in partial differential equations, namely when $Z$ is the space of infinitely-differentiable test functions with compact support and endowed with the topology used in the theory of distributions, which is the inductive limit of a countable family of separable Fr\'echet spaces \cite[Section I.1]{Yosida1980}. Notice that any separable Fr\'echet space is a Polish space. Now, if $Z=\cup_{i\in\mathds N} Z_i$ is any inductive limit of a countable sequence of subsets $Z_i$ which are separable Fr\'echet spaces, then any open set $O$ in $Z$ can be written as $O = \cup_i (O \cap Z_i)$, and each $O\cap Z_i$ is open in $Z_i$. Then, if $X$ is either a metrizable space or a separable Fr\'echet space endowed with its weak topology, we have that each $Z_i$ is also continuously injected into $X$, and, by the previous results, each $O\cap Z_i$ is Borel in $X$. Thus, the countable union $O = \cup_i (O\cap Z_i)$ is also Borel in $X$. Hence, we deduce that any Borel subset of $Z$ is a Borel subset of $X$.

Finally, we should mention that the results described above can be extended to include products of those spaces, allowing us to tackle systems of equations. 

The conditions on $Z$ and $X$ described above certainly do not exhaust all the possible situations (see \cite{RogersJayne1980} for further conditions). For this reason, our main results on the existence of statistical solutions in the phase space (Theorems \ref{thm_projectedss_is_ss}, \ref{thm-existence-projectedss}, and \ref{thmextensionhdoubleprimephasespace}, and Proposition \ref{prop-energy-ineq}) do not assume any extra structure on the spaces $Z$ and $X$, instead just assume that Borel subsets of $Z$ are Borel subsets of $X$, leaving the verification of this condition to the applications.

\subsection{Topologies for measure spaces and related results}
\label{subsectopformeasspaces}

In \cite{Topsoebook}, Topsoe considered a topology in $\MX$ obtained as the smallest one for which the mappings $\mu \mapsto \mu(f)$ are upper semicontinuous, for every bounded and upper semi-continuous real-valued function $f$ on $X$. Topsoe calls this topology the ``weak topology'', but in order to avoid any confusion we call it here the \textbf{weak-star semi-continuity topology} on $\MX$. When a net $\{\mu_\alpha\}_\alpha$ converges to $\mu$ with respect to this topology, we write $\mu_\alpha \wconv \mu$.

A more common topology used in $\MX$ is the \textbf{weak-star topology,} which is the smallest topology for which the mappings $\mu \mapsto \mu(f)$ are continuous, for every bounded and continuous real-valued function $f$ on $X$. If a net $\{\mu_\alpha\}_\alpha$ converges to $\mu$ with respect to this topology, we denote $\mu_\alpha \wsconv \mu$.

It is not difficult to see that the weak-star semi-continuity topology is equivalent to the topology obtained as the smallest one for which the mappings $\mu \mapsto \mu(f)$ are lower semi-continuous, for every bounded and lower semi-continuous real-valued function $f$ on $X$. Then, considering now $f$ as a bounded and continuous real-valued function on $X$, it follows that, for every $a,b \in \mathds R$, the set
\begin{multline}
\{\mu \in \MX \,|\, \mu(f) \in (a,b)\} = \\
\{\mu \in \MX \,|\, \mu(f) \in (a,+\infty)\} \cap \{\mu \in \MX \,|\, \mu(f) \in (-\infty,b)\}
\end{multline}
is open in the weak-star semi-continuity topology, since $f$ is in particular a lower and upper semi-continuous function. This shows that the weak-star topology is in general weaker than the weak-star semi-continuity topology. Moreover, according to Lemma \ref{lemportmanteau} below, if $X$ is a completely regular Hausdorff space, then these two topologies coincide when restricted to the space $\MXt$.

The following lemma summarizes the properties about the weak-star semi-continuity topology that we have mentioned and provides some additional useful characterizations (see \cite[Theorem 8.1]{Topsoebook}).

\begin{lem}\label{lemportmanteau}Let $X$ be a Hausdorff space.
For a net $\{\mu_\alpha\}_\alpha$ in $\MX$ and $\mu\in \MX$,
consider the following statements:
\begin{itemize}
\item [(1)]$\mu_\alpha \wconv\mu$;
\item [(2)]$\limsup \mu_\alpha (f)\leq \mu(f)$, for all bounded upper
semicontinuous function $f$;
\item [(3)]$\liminf \mu_\alpha (f)\geq \mu(f)$, for all bounded lower
semicontinuous function $f$;
\item [(4)]$\lim_\alpha \mu_\alpha(X)=\mu(X)$ and $\limsup \mu_\alpha (F)\leq \mu(F)$, for all closed set $F\subset X$;
\item [(5)]$\lim_\alpha \mu_\alpha(X)=\mu(X)$ and $\liminf \mu_\alpha (G)\geq \mu(G)$, for all open set $G\subset X$;
\item [(6)]$\mu_\alpha \wsconv \mu$.
\end{itemize}

Then the first five statements are equivalent and each of them implies the last one.

Furthermore, if $X$ is also completely regular and $\mu \in \MXt$, then all six statements are equivalent.
\end{lem}

Although our framework is based on a general Hausdorff space, our proofs rely on reducing some structures to compact subsets, hence completely regular Hausdorff spaces, i.e., Tychonoff spaces. Moreover, since the measures in our proofs are usually tight, then in this setting both topologies coincide, so that we could have very well considered only the weak-star topology. However, we prefer to use the weak-star semi-continuity topology since it is a more natural topology for arbitrary Hausdorff spaces which simplifies our presentation and parts of the proofs, and which yields a compactness result in a stronger topology.

When dealing with convergent nets in a given space, a natural question arises as to whether the limits are unique. This requires the given space of measures to be Hausdorff. The delicate issue is to determine the minimal hypotheses for that.

If $X$ is a metrizable topological space, then $\MX$ is a Hausdorff space with respect to the weak-star topology (see \cite[Section 15.1]{AB}), and hence also with respect to the weak-star semi-continuity topology. However, requiring $X$ to be metrizable is too restrictive for our purposes. In looking for a more general setting for the space $X$, we were led to work within the space of tight measures $\MXt$ and with the weak-star semi-continuity topology, which Topsoe proved to be a Hausdorff space. This key result was in fact a motivation for Topsoe to advance his work on the subject (see \cite[Preface]{Topsoebook}). A proof of this fact is given in \cite[Theorem 11.2]{Topsoebook} by showing that the limits of convergent nets are unique. Here we chose to include a proof showing directly that distinct measures in $\MX$ can be separated by open sets.

\begin{thm}\label{thmMXthausdorff}
Let $X$ be a Hausdorff space. Then, $\MXt$ is a Hausdorff space with respect to the weak-star semi-continuity topology.
\end{thm}

\begin{proof}
Consider two distinct measures $\mu_1,\mu_2$ in $\MXt$.

Suppose at first that $\mu_1(X) \neq \mu_2(X)$. Let us assume, without loss of generality, that $\mu_1(X) < \mu_2(X)$. Then there exists $\varepsilon > 0$ such that
\be\label{eqmu1neqmu2}
\mu_1(X) < \mu_1(X) + \varepsilon < \mu_2(X) - \varepsilon < \mu_2(X).
\ee
Denote $a = \mu_1(X) + \varepsilon$ and $b = \mu_2(X) - \varepsilon$. Let $f$ be the constant function $f \equiv 1$ on $X$. Then, it follows from \eqref{eqmu1neqmu2} that
\be\label{lemHausdeq1}
\mu_1 \in \{\mu \in \MX \,|\, \mu(f) \in (-\infty,a)\}
\ee
and
\be\label{lemHausdeq2}
\mu_2 \in \{ \mu \in \MX \,|\, \mu(f) \in (b,+\infty)\}.
\ee
Since $f$ is in particular a bounded and continuous real-valued function on $X$, it follows that the sets in \eqref{lemHausdeq1} and \eqref{lemHausdeq2} are open sets in $X$. These sets are also clearly disjoint. Hence, $\mu_1$ and $\mu_2$ can be separated by disjoint open sets in $X$.

Now suppose that $\mu_1(X) = \mu_2(X)$. We claim that there exists a set $A \in \fB_X$ such that $\mu_1(\bar{A}) < \mu_2(\mathring{A})$, where $\bar{A}$ and $\mathring{A}$ denote the closure and interior of $A$, respectively. Indeed, suppose by contradiction that
\be\label{eqcontradict}
\mu_1(\bar{A}) \geq \mu_2(\mathring{A}), \quad \forall A \in \fB_X.
\ee

Consider $E \in \fB_X$ and let $K_1,K_2$ be arbitrary compact sets in $X$ satisfying $K_1 \subset X \backslash E$ and $K_2 \subset E$. Then, since $X$ is a Hausdorff space, there exist disjoint open sets $B_1,B_2$ in $X$ such that $K_1 \subset B_1$ and $K_2 \subset B_2$.

In particular, using that $B_2$ is an open set, it follows from \eqref{eqcontradict} that
\be\label{eqmu2}
\mu_2 (B_2) = \mu_2 (\mathring{B_2}) \leq \mu_1 (\bar{B_2}).
\ee

But clearly $\bar{B_2} \subset X \backslash K_1$. Then,
\be\label{eqmu1}
\mu_1(\bar{B_2}) \leq \mu_1(X \backslash K_1) = \mu_1(X) - \mu_1 (K_1).
\ee

From \eqref{eqmu2} and \eqref{eqmu1}, we obtain that
\[
\mu_2(K_2) \leq \mu_2 (B_2) \leq \mu_1(X) - \mu_1(K_1).
\]

Now since $K_1$ and $K_2$ were chosen arbitrarily, taking the supremum over all compact sets $K_1,K_2$ with $K_1 \subset X \backslash E$ and $K_2 \subset E$, and using that $\mu_1$ and $\mu_2$ are tight, it follows that
\[
\mu_2(E) \leq \mu_1(X) - \mu_1 (X \backslash E) = \mu_1 (E).
\]
Thus,
\be\label{eqleq}
\mu_2(E) \leq \mu_1(E), \quad \forall E \in \fB_X.
\ee
This implies in particular that
\[
\mu_2 (X) - \mu_2(E) = \mu_2( X \backslash E) \leq \mu_1 (X \backslash E) = \mu_1(X) - \mu_1(E).
\]
Using the hypothesis that $\mu_1(X) = \mu_2(X)$, we then obtain
\be\label{eqgeq}
\mu_2(E) \geq \mu_1(E), \quad \forall E \in \fB_X.
\ee
Now \eqref{eqleq} and \eqref{eqgeq} yield $\mu_1 = \mu_2$, which is a contradiction.

Thus, we may consider $A \in \fB_X$ such that $\mu_1(\bar{A}) < \mu_2(\mathring{A})$. The argument now follows analogously to the previous case. Consider $\varepsilon > 0$ satisfying
\[
\mu_1(\bar{A}) < \mu_1( \bar{A}) + \varepsilon < \mu_2(\mathring{A}) - \varepsilon < \mu_2(\mathring{A})
\]
and denote $a' = \mu_1( \bar{A}) + \varepsilon$ and $b' = \mu_2(\mathring{A}) - \varepsilon$. Since the characteristic function $\chi_{\bar{A}}$ of $\bar{A}$ is bounded and upper semi-continuous and the characteristic function $\chi_{\mathring{A}}$ of $\mathring{A}$ is bounded and lower semicontinuous, then the sets
\[
\{\mu \in \MX \,|\, \mu(\chi_{\bar{A}}) \in (-\infty, a')\}
\]
and
\[
\{\mu \in \MX \,|\, \mu(\chi_{\mathring{A}}) \in (b', +\infty)\}
\]
are clearly disjoint open sets in $X$ containing $\mu_1$ and $\mu_2$, respectively. Thus, $\mu_1$ and $\mu_2$ can also be separated by disjoint open sets of $X$ in case $\mu_1(X) = \mu_2(X)$.

This proves that $\MXt$ is a Hausdorff space.
\end{proof}

Moreover, if $X$ is assumed to be a completely regular Hausdorff space, then by Lemma \ref{lemportmanteau} the weak-star semi-continuity and weak-star topologies are the same in $\MXt$. Thus, it follows from Theorem \ref{thmMXthausdorff} that $\MXt$ is also Hausdorff with respect to the weak-star topology. We have just proved the following corollary:

\begin{cor}\label{corMXthausdorff}
Let $X$ be a completely regular Hausdorff space. Then $\mM(X,\text{tight})$ is a Hausdorff space with respect to the weak-star topology.
\end{cor}

\begin{rmk}
If we relaxed the hypothesis in Corollary \ref{corMXthausdorff} by assuming $X$ to be a regular space instead of a completely regular space, then this result would no longer be valid. Indeed, it is possible to construct examples of regular spaces which are not completely regular and containing two distinct points $a,b$ for which every continuous real-valued function satisfies $f(a)=f(b)$ (see \cite{Mysior,Raha}). The corresponding space $\mM(X,\text{tight})$ is then not Hausdorff with respect to the weak-star topology, since it suffices to consider the Dirac measures $\delta_a$ and $\delta_b$ concentrated on $a$ and $b$, respectively, and to note that $\delta_a(f)=\delta_b(f)$, for every bounded and continuous real-valued function $f$ on $X$.
\end{rmk}

We next state a result of compactness on the space of tight measures $\MXt$ that is essential for our main result. For a proof of this fact, see \cite[Theorem 9.1]{Topsoebook}.

\begin{thm}\label{thmMXtcompactness}
Let $X$ be a Hausdorff space and let $\{\mu_{\alpha}\}_{\alpha}$ be a net in
$\MXt$ such that $\limsup \mu_{\alpha} (X) < \infty$. If $\{\mu_{\alpha}\}_{\alpha}$ is uniformly tight, then it is compact with respect to the weak-star semi-continuity topology in $\MXt$.
\end{thm}

The previous theorem allows us to obtain a convergent subnet of a given net in $\MXt$, provided it satisfies the required conditions. Also, by using Theorem \ref{thmMXthausdorff}, we can guarantee that the limit of this convergent subnet is unique.

Evidently, all the results shown above are also valid in the space of probability measures. In what follows, we consider both the spaces of probability measures defined over the Hausdorff space $X$ and over the space of continuous paths $\mX$.

\subsection{Weak-star scalarwise derivative}
\label{subsecweakderivatives}

For an evolutionary differential equation, we need a notion of time-derivative for functions with values in a vector space. For our purposes we only need a scalarwise-type derivative, but since we look for solutions which are also weakly continuous (see Section \ref{subsecevoleq}), this turns out to be a functional derivative associated with the functional integral in the sense of Pettis (see Remark \ref{weakscalarwiseintegralandPettis}). 

We follow the nomenclature of scalarwise property given in \cite{Bourbaki1959, Edwards1965}, except that we restrict ourselves to functionals with values in the dual of a topological vector space.  More precisely, we consider a topological vector space $Y$ and its dual $Y'$, and we say that $u: I \rightarrow Y'$ has a certain property $P$ \textbf{weak-star scalarwise} if each function $t\mapsto \langle u(t),v\rangle_{Y',Y}$ has the property $P$, for every $v\in Y$. We will be interested in the weak-star scalarwise versions of the properties of a function being (Lebesgue) measurable, integrable, locally integrable, continuous, absolutely continuous, and almost everywhere differentiable. 

Saying that $u$ is weak-star scalarwise continuous is the same as saying that $u$ is weak-star continuous (i.e $u$ is continuous as a function from $I$ into the topological vector space $Y_{\rw^*}'$). But the same cannot be said about the other properties. For instance, $u$ can be weak-star scalarwise measurable without being Borel measurable from $I$ into $Y_{\rw^*}'$ (see Remark \ref{wscalarmeasnotstrongmeas}).

We are interested in functions $u: I \rightarrow Y'$ that are weak-star scalarwise  continuous and such that there exists a weak-star scalarwise locally integrable function $w: I \rightarrow Y'$ which is the weak-star scalarwise derivative of $u$ in the distribution sense on $I$, i.e.
\begin{equation}
  \label{evoleqweakdist}
 - \int_I \psi'(t)\langle u(t),v\rangle_{Y',Y}  \;\rd t = \int_I \psi(t) \langle w(t), v \rangle_{Y' ,Y} \;\rd t, \quad \forall v\in Y, \;\forall \psi\in \mC_\rc^\infty(I).
\end{equation}
This is equivalent  to assuming that $u: I \rightarrow Y'$ is weak-star scalarwise  continuous and that there exists a weak-star scalarwise locally integrable function $w: I \rightarrow Y'$ satisfying
\begin{equation}
  \label{evoleqintegral}
  \langle u(t_2),v\rangle_{Y',Y} - \langle u(t_1),v\rangle_{Y',Y} = \int_{t_1}^{t_2} \langle w(t), v \rangle_{Y' ,Y} \;\rd t, \quad \forall t_1, t_2\in I, \; \forall v\in Y.
\end{equation}
This is also equivalent (see e.g. \cite[Theorem 3.35]{Folland}) to considering functions that are weak-star scalarwise absolutely continuous and such that there exists a weak-star scalarwise locally integrable function $w: I \rightarrow Y'$ satisfying
\begin{equation}
  \label{evoleqweakae}
  \frac{\rd}{\rd t}\langle u(t),v\rangle_{Y',Y} = \langle w(t), v \rangle_{Y' ,Y}, \quad \forall v\in Y, \; \text{for a.e. } t \in I.
\end{equation}

Using the shorter property \eqref{evoleqweakae} of the equivalent forms \eqref{evoleqweakdist}, \eqref{evoleqintegral}, and \eqref{evoleqweakae}, we define the space
\begin{equation}
  \label{defYprime}
  \mY_1=\left\{ u: I \rightarrow Y' : \parbox{3.8in}{$u$ is weak-star scalarwise absolutely continuous and \smallskip \\ 
   \hspace{2mm} $\exists$ $w:I \rightarrow Y'$ weak-star scalarwise locally integrable \smallskip \\
   \ \ \ with $\displaystyle \frac{\rd}{\rd t}\langle u(t),v\rangle_{Y',Y} = \langle w(t), v \rangle_{Y' ,Y}$, a.e.  $t\in I$, $\forall v\in Y$} \right\}.
\end{equation}

The weak-star scalarwise derivative $w$ of an element of $\mY_1$, as given in the definition \eqref{defYprime}, is uniquely defined in the case that $Y_\rw$ is separable (see Remark \ref{wunique}), but may not be unique in general (see Remark \ref{wnotunique}).

\begin{rmk}
  \label{weakscalarwiseintegralandPettis}
  If $w:I \rightarrow Y'$ is weak-star scalarwise integrable on $I$, then $w$ induces a linear map $\Lambda_w$ belonging to the algebraic dual of $Y$, associating each element $v\in Y$ to the scalar $\Lambda_w(v) = \int_I \langle w(t),v\rangle_{Y',Y} \;\rd t$. This map is uniquely defined. If this map belongs to $Y'$, then it is the Pettis integral of $w$ over $I$ and it is written as $\Lambda_w = \int_I w(t) \;\rd t$. Similarly for the locally weak-star scalarwise integrable case, when integrating over compact measurable sets within $I$.  More generally, one may consider integrals on any topological vector space, not necessarily dual spaces, but this is the context we are interested in. There are a number of results concerning the scalarwise measurability and integrability and conditions for the existence of the Pettis integral; see e.g. \cite{Edwards1965, Bourbaki1959} for further details. For instance, if $Y$ is a locally convex topological vector space with the so-called GDF property (i.e. any linear map of $Y$ into a Banach space $B$ having a graph that is sequentially closed in $Y\times B$ is a continuous map), then any (weak-star) scalarwise integrable function $w$ from $I$ into $Y'$ has a Pettis integral well defined and belonging to $Y'$ \cite[Theorem 8.16.1]{Edwards1965}. Any Fr\'echet space and any inductive limit of Fr\'echet spaces have the GDF property \cite[Subsection 8.14.12]{Edwards1965}. In our work, however, we do not need to worry about conditions for the existence of the Pettis integral in $Y'$, since the weak formulation is all we need. Nevertheless, the Pettis integral of any weak-star scalarwise derivative of a function in $\mY_1$ does exist since the equivalent formulation \eqref{evoleqintegral} shows that the weak-star scalarwise integral is a continuous linear function on $Y$.
\end{rmk}

\begin{rmk}
  \label{wunique}
  In the case that $Y_\rw$ is separable, i.e. if $Y$ is separable under the weak topology, then $w: I \rightarrow Y'$ satisfying the conditions in the definition of $\mY_1$ in \eqref{defYprime} is unique (up to measure zero on $I$). Indeed, suppose $w_1$ and $w_2$ are two functions satisfying the conditions on $w$ in the definition of $\mY_1$. Assume there exists a countable set $D=\{v_j\}_{j\in \mathds N}$ in $Y$ which is dense in $Y$ with respect to the weak topology. Then, for each $v_j$, we have $\langle w_1(t), v_j \rangle_{Y' ,Y}= \langle w_2(t), v_j \rangle_{Y' ,Y}$, for almost every $t$ in $I$. Since the set $D$ is countable, there is a set of full measure $I_0$ in $I$ for which $\langle w_2(t)-w_1(t), v_j \rangle_{Y' ,Y}= 0$, $\forall t\in I_0$, $\forall j\in \mathds N$. For each $t\in I_0$, since $D$ is dense in $Y$ in the weak topology, we find that $\langle w_2(t)-w_1(t), v \rangle_{Y' ,Y}= 0$, for all $v\in Y$. This means that $w_2(t) = w_1(t)$, for all $t\in I_0$. Since $I_0$ is of full measure, we find that $w_1=w_2$ up to a set of measure zero. 
\end{rmk}

\begin{rmk}
  \label{wnotunique}  
  If $Y_\rw$ is not separable, then we cannot guarantee that $w: I \rightarrow Y'$ satisfying the conditions in the definition of $\mY_1$ in \eqref{defYprime} is unique. Indeed, take any nonseparable Hilbert space $Y$ with a basis with the cardinality of the continuum. Denote the inner product by $\langle \cdot, \cdot \rangle_Y$ and identify the dual $Y'$ with $Y$ itself. Let $I=[0,1]$ and let $\{v_s\}_{s\in I}$ be an orthonormal basis, parametrized by the continuum $I$. Consider the function $\tilde w:I\rightarrow Y$ given by $\tilde w(t)=v_t$, for $t\in [0,1]$. For each $v_s$, the real-valued function $t\mapsto \langle \tilde w(t),v_s\rangle_Y$ is equal to $0$, for $t\neq s$, and to $1$, for $t=s$. Thus, it is almost everywhere equal to zero and, in particular, Lebesgue integrable on $I$. If $v\in Y$ is arbitrary, then $v$ is a linear combination of at most a countable set of indices, so that $\langle \tilde w(t),v\rangle_Y$ is different from zero at most at a countable set of values $t$. This means that the real-valued function $t\mapsto \langle \tilde w(t),v\rangle_Y$ is, again, zero almost everywhere and Lebesgue integrable. Thus,
\[
  \int_0^1 \langle \tilde w(t),v\rangle_Y \; \rd t = 0, \quad \forall v\in Y.
\]
Nevertheless, $\tilde w$ is not zero almost everywhere in $Y$. In fact, it is different from zero everywhere on $I$. Thus, if a function $w:I\rightarrow Y'$ satisfies the conditions in the definition of $\mY_1$ in \eqref{defYprime}, then so does $w+\tilde w$, which shows the nonuniqueness of the weak-star scalarwise derivative in this case.
\end{rmk}  

\begin{rmk}
  \label{wscalarmeasnotstrongmeas}
  The same function $w$ defined in Remark \ref{wnotunique} when $Y$ is a nonseparable Hilbert space with an orthonormal basis with the cardinality of the continuum is an example of a function which is (weak-star) scalarwise Borel measurable on $I$ but not Borel measurable from $I$ into $Y_\rw$ (which in this case is exactly $Y_{w*}'$). Indeed, for each $t\in I$, let $O_t$ be a weakly open set in $Y$ which only contains the vector $v_t$ of the basis, not $v_s$, for $s\neq t$ (say $O_t = \{v\in Y: \;\langle v, v_t\rangle_Y > 1/2\}$). Take any nonmeasurable subset $A$ of $I$. Then, $O=\cup_{t\in A} O_t$ is a weakly open set in $Y$ such that $w^{-1}(O)=A$ is not measurable.
\end{rmk}

\begin{rmk}
  \label{bochnerint}
  In case $Y$ is a Banach  space, then so is $Y'$, and one may consider the Bochner integral in  $Y'$ \cite{Yosida1980}. The set of locally Bochner integrable functions that have a weak derivative which is also locally Bochner integrable in $Y'$ forms the space $W^{1,1}_{\rloc}(I;Y')$, and, in this case, the time derivative $u_t$ of a function $u\in W^{1,1}_{\rloc}(I;Y')$ is uniquely defined and satisfies $u(t_1) = u(t_0) + \int_{t_0}^{t_1} u_t(t)\;\rd t$ in $Y'$, for almost every $t_0,t_1\in I$, where the time integral is the Bochner integral in $Y'$. In this case, we notice that $W^{1,1}_{\rloc}(I;Y')$ is included in the space $\mY_1$ defined in \eqref{defYprime} (see e.g. \cite[Lemma 3.1.1]{temam84}), with \eqref{evoleqweakdist}, \eqref{evoleqintegral}, and \eqref{evoleqweakae} holding with $w=u_t$. Therefore, the framework above can be used in applications where the differential equation holds in a functional sense associated with the Bochner integral.
\end{rmk}

\subsection{Evolution equations}
\label{subsecevoleq}

The statistical solutions in a phase space $X$ (see Definition \ref{def-statsolphasespace}) are directly related to an evolution equation of the form
\be\label{evoleq}
u_t(t)=F(t,u(t)),
\ee
where the unknown $u$ belongs to the space $\mX=\Cloc(I,X)$, with $X$ being a Hausdorff space, and $I\subset \mathds R$, an interval. 

We want to be able to make sense of the equation \eqref{evoleq} under minimal hypotheses on the structure of $F$ and on the spaces involved. Of course, the peculiarities of the solutions are of interest in each application, but for our general result of existence of statistical solutions for the initial value problem, we want to avoid superfluous conditions. 

For partial differential equations, the right hand side of the equation \eqref{evoleq} involves spatial derivatives of the solution $u$, so it is not expected that $F$ be defined on $X$, but, instead, on a more regular space included in $X$. Moreover, for the time-derivative to be defined, we need a vector space structure, so that $u$ needs to be included in some topological vector space, possibly less regular than $X$.

With that in mind, we consider another Hausdorff space $Z$ and a topological vector space $Y$, and consider $F$ such that
\begin{equation}
  \label{typefunctionF}
  F:I\times Z \rightarrow Y',
\end{equation}
with the spaces satisfying
\begin{equation}
  \label{spacesforstatsolphasespace}
   Z \subset X \subset Y_{\rw^*}',
\end{equation}
where each space is continuously included in the next one, and where $Y_{\rw^*}'$ is the dual space $Y'$ endowed with the weak-star topology. The reason for the inclusion $X\subset Y_{\rw^*}'$ is that we will in fact consider the weak formulation of the evolution equation \eqref{evoleq}, so that \eqref{spacesforstatsolphasespace} is the natural assumption to make. In many cases, $Z=Y$, but this is not always convenient or necessary (see Section \ref{subsecreacdiffeq} and Remark \ref{ZequalsYrmk}).

Then, we look at the weak formulation of \eqref{evoleq}, i.e.
\begin{equation}
  \label{evoleqweak}
  \frac{\rd}{\rd t}\langle u(t),v\rangle_{Y',Y} = \langle F(t,u(t)), v \rangle_{Y' ,Y},
\end{equation}
in the distribution sense on $I$, for any $v\in Y$.

Since $u\in \mX=\Cloc(I,X)$ and $X\subset Y_{\rw^*}'$, we have that $u\in \Cloc(I,Y_{\rw^*}')$, so that $t\mapsto \langle u(t),v\rangle_{Y',Y}$ is continuous, for every $v\in Y$. Thus, in order for \eqref{evoleqweak} to make sense, $t\mapsto \langle u(t),v\rangle_{Y',Y}$ needs to be absolutely continuous and the right hand side of \eqref{evoleqweak} needs to be integrable. With that in mind, we consider $u$ belonging also to the set $\mY_1$ defined in \eqref{defYprime}. Moreover, we need $u$ such that $u(t)\in Z$ for almost every $t\in I$. Therefore, we define the spaces
\begin{equation}
  \label{defZ}
  \mZ=\{u\in \mX: u(t)\in Z \mbox{ for almost all } t\in I\},
\end{equation}
and
\begin{equation}
  \label{defXprime}
  \mX_1 = \mZ \cap \mY_1.
\end{equation}
Then, if $u\in \mX_1$ is such that
\begin{equation}
  \label{Ftuislocint}
  t \mapsto F(t,u(t)) \text{ is weak-star scalarwise locally integrable,}
\end{equation}
as a function from $I$ into $Y'$, it makes sense to require that $u_t = F(t,u)$ be valid in the weak sense \eqref{evoleqweak}. Notice that, for $u\in\mX_1$ such that \eqref{Ftuislocint} holds, condition \eqref{evoleqweak} is precisely \eqref{evoleqweakdist} with $w(t)=F(t,u(t))$ and is equivalent to both \eqref{evoleqintegral} and \eqref{evoleqweakae}, with the same $w$.

\subsection{Cylindrical test functions}

Consider a topological vector space $Y$ and let $v_1,\ldots, v_k \in Y$, where $k\in \mathds N$. Let $\phi$ be a continuously differentiable real-valued function on $\mathds R^k$ with compact support. For each $u\in Y'$, define $\Phi(u)\in \mathds R$ by
\[  \Phi(u) = \phi(\langle u, v_1\rangle_{Y',Y}, \ldots, \langle u,
v_k\rangle_{Y',Y}).
\]
The function $\Phi$ is clearly continuous from $Y'$ to $\mathds R$ and in fact it is differentiable in $Y'$, with differential $\Phi'$ at $u\in Y'$ given by
\be\label{eqfrechetderphi}
 \Phi'(u)=\sum_{j=1}^k\partial_j\phi(\langle u, v_1\rangle_{Y',Y}, \ldots,
\langle u, v_k\rangle_{Y',Y}) v_j,
\ee
where $\partial_j\phi$ denotes the derivative of $\phi$ with respect to its $j$-th coordinate. For each $u, w\in Y'$, 
\[ \langle w, \Phi'(u)\rangle_{Y',Y} = \sum_{j=1}^k \partial_j\phi(\langle u, v_1\rangle_{Y',Y}, \ldots,
\langle u, v_k\rangle_{Y',Y}) \langle w, v_j\rangle_{Y',Y}
\]
is the G\^ateaux derivative of $\Phi$ at $u$, in the direction of $w$. Moreover, $\Phi'(u)$ belongs to $Y$ and the differential $\Phi'$ is continuous from $Y_{\rw^*}'$ into itself, where we recall that $Y'_{w*}$ denotes the space $Y'$ endowed with its weak-star topology. If $Y$ is a Banach space, then $Y'$ is a Banach space under the strong operator norm, and $\Phi'$ is the Fr\'echet differential of $\Phi$ in this strong norm of $Y'$.

Functions of this form are called \textbf{cylindrical test functions in} $Y'$ and play an important role as test functions in the definition of statistical solution in phase space (see Definition \ref{def-statsolphasespace}). In that context, we will also consider a Hausdorff space $X$ which is assumed to be continuously imbedded in $Y'_{w*}$. Notice that since $u \mapsto \langle u, v\rangle_{Y',Y}$ is continuous in the weak-star topology for any $v\in Y$, the function $\Phi$ is also continuous from $Y'_{w*}$ into $\mathds R$. Then, since $X$ is continuously imbedded in $Y'_{w*}$, we may consider $\Phi$ restricted to $X$, which is continuous as a function from $X$ into $\mathds R$.

\begin{rmk}
The set of cylindrical test functions is a relatively large set, as can be seen from the Stone-Weierstrass Theorem (see e.g. \cite[Theorem IV.6.16]{dunfordschwartz}). In fact, consider a compact set $K'$ in $Y_{\rw^*}'$ (e.g. a (strongly) closed ball in $Y'$, if $Y$ is a normed space or, more generally, any closed subset of $Y_{\rw^*}'$ which is weakly bounded and equicontinuous (\cite[Theorem 1.11.4]{Edwards1965})). Denote by $\mS_{K'}$ the collection of the real-valued functions on $K'$ which are the restriction to $K'$ of the cylindrical test functions. Clearly, $\mS_{K'}\subset \mC(K')$ and, if $\Psi_1,\Psi_2\in \mS_{K'}$, then their sum $\Psi_1+\Psi_2$ and their product $\Psi_1\Psi_2$ also belong to $\mS_{K'}$. This means that $\mS_{K'}$ is a subalgebra of the algebra $\mC(K')$. Choose any $v\in Y$ and notice that the set $J = \{\langle u, v\rangle_{Y',Y}: u\in K'\}$ is compact in $\mathds R$, otherwise $K'$ would not be weak-star compact. Then, by taking any $\phi:\mathds{R}\rightarrow \mathds{R}$ which is continuously differentiable and is equal to $1$ on $J$, we see that $\Phi(u) = \phi(\langle u, v\rangle_{Y',Y})$ is a cylindrical test function which is equal to $1$ for $u\in K'$, showing that $\mS_{K'}$ contains the unit element. Moreover, if $u_1, u_2$ are distinct points in $K'$, there exists $v\in Y$ such that $\langle u_1, v\rangle_{Y',Y} \neq \langle u_2, v\rangle_{Y',Y}$. Hence, by choosing a continuously differentiable function $\phi:\mathds{R}\rightarrow \mathds{R}$ which is compactly supported and assumes different values at the points $\langle u_1, v\rangle_{Y',Y}$ and $\langle u_2, v\rangle_{Y',Y}$, we see that $\Phi(u) = \phi(\langle u, v\rangle_{Y',Y})$ is a cylindrical test function which assumes different values at $u_1$ and $u_2$, proving that $\mS_{K'}$ separates the points in $K'$. Therefore, the Stone-Weierstrass Theorem yields that $\mS_{K'}$ is dense in $\mC(K')$, for the uniform topology. Similarly, since $X$ is continuously imbedded into $Y_{w*}'$, it follows that, for any compact subset $K$ of $X$, the collection of the real-valued functions on $K$ which are the restrictions to $K$ of cylindrical test functions is dense in $\mC(K)$.
\end{rmk}

\subsection{The Nemytskii operator}

In relation to the evolution equation \eqref{evoleq}, we consider a function $F$ of the form $F:I\times Z\rightarrow W$, where $Z$ and $W$ are Hausdorff topological spaces, and consider the operator $G: I\times \mZ \rightarrow W$ given by $G(t,u) = F(t,u(t))$, where $\mZ$ is defined in \eqref{defZ}. It is of fundamental importance to deduce measurability properties for $G$ from those of $F$. This is the aim of this section. This extended operator $G$ is known as a \textbf{Nemytskii operator} in the context of partial differential equations, where the Hausdorff spaces above are simply spaces of real-valued functions defined on subdomains of Euclidean spaces. In this regard, we have the following results.

\begin{lem}
  \label{lemnemytskiiX}
Let $X$ and $W$ be Hausdorff spaces and let $I$ be an interval in $\mathds R$. Consider a function $F : I \times X \rightarrow W$ and define the associated Nemytskii operator $G:I\times \mX\rightarrow W$ by $G(t,u) = F(t,u(t))$, for $t\in I$ and $u\in \mX=\Cloc(I,X)$. If $F$ is a $(\fL_I\otimes\mathcal\fB_X,\fB_W)$-measurable function, then $G$ is $(\fL_I\otimes\fB_{\mX},\fB_W)$-measurable.
\end{lem}

\begin{proof}
Consider the projection
\begin{eqnarray*}
\qquad\Pi_I: I\times \mX&\rightarrow &I\\
                (t,u)&\mapsto& t,
\end{eqnarray*}
and the evaluation operator
\begin{eqnarray}\label{eqdefUonIX}
\qquad U: I\times \mX&\rightarrow &X \nonumber\\
                (t,u)&\mapsto& u(t).
\end{eqnarray}
The projection $\Pi_I$ is clearly a $(\fL_I \otimes \fB_{\mX},\fL_I)$-measurable function. Moreover, $U$ is a continuous function and then, in particular, $(\fB_{I\times \mX},\fB_X)$-measurable. Since $\mX$ is a Hausdorff space and $I$ is a second countable  Hausdorff space it follows that $\fB_{I\times\mX} = \fB_I \otimes \fB_{\mX}$ (\cite[Lemma 6.4.2]{Bogachev}). Thus, $U$ is also a $(\fL_I \otimes \fB_{\mX},\fB_X)$-measurable function. From this we obtain that the function $(\Pi_I,U): I \times \mX \rightarrow I\times X$, defined by $(\Pi_I,U)(t,u)=(t,u(t))$, is $(\fL_I\otimes\fB_{\mX},\fL_I\otimes \fB_X)$-measurable.
Now, since $F$ is, by assumption, $(\fL_I\otimes \fB_X,\fB_W)$-measurable, and $G$ can be written as $G=F\circ(\Pi_I, U)$ we conclude that $G$ is  $(\mathcal L_I\otimes\fB_{\mX},\fB_W)$-measurable.
\end{proof}

Now we consider the measurability of the trivial extension of a function $F$ defined on $I\times Z$ to $I\times X$.

\begin{lem}
  \label{lemextF}
  Let $Z, X, W$ be (nonempty) Hausdorff spaces such that $Z\subset X$, with continuous inclusion, and let $I$ be an interval in $\mathds R$. Consider a function $F : I \times Z \rightarrow W$. Let $w$ be an arbitrary element in $W$ and define the extension $\tilde F:I\times X\rightarrow W$ by
\be\label{defextF}
\tilde F(t,u)= \left\{\begin{array}{ll}
                  F(t,u), &\mbox{ if } u\in Z,\\
                  w, &\mbox{otherwise}.
                 \end{array}\right.
\ee
Assume that every Borel subset of $Z$ is a Borel subset of $X$ and that $F$ is a $(\fL_I\otimes\fB_Z,\fB_W)$-measurable function. Then $\tilde F$ is a $(\fL_I\otimes\fB_X,\fB_W)$-measurable function.
\end{lem}

\begin{proof}
Let $E\in \fB_W$ and note that
\[\tilde F^{-1}(E)=\left\{\begin{array}{ll}
     F^{-1}(E)\cup (I\times(X\setminus Z)), &\mbox{ if } w\in E,\\
     F^{-1}(E), &\mbox{ if } w\notin E.
\end{array}\right.\]
Since $F$ is $(\fL_I\otimes\fB_Z,\fB_W)$-measurable then $F^{-1}(E)\in
\fL_I\otimes\fB_Z$. From the hypothesis that $\fB_Z\subset \fB_X$, this implies that $F^{-1}(E)\in
\fL_I\otimes\fB_X$. Moreover, again from $\fB_Z\subset \fB_X$, we have, in particular, that $Z$ is Borel in $X$, so that $X\setminus Z\in \fB_X$, as well, and, hence, $I\times (X\setminus Z) \in \fL_I\otimes\fB_X$. Thus, $\tilde F^{-1}(E)\in \fL_I\otimes \fB_X$, regardless of whether $w$ belongs to $E$ or not. This proves that $\tilde F$ is $(\fL_I\otimes\fB_X,\fB_W)$-measurable.
\end{proof}

Combining Lemmas \ref{lemnemytskiiX} and \ref{lemextF} we obtain the following result.

\begin{prop}
  \label{propnemytskii}
Let $Z, X, W$ be (nonempty) Hausdorff spaces such that $Z\subset X$, with continuous inclusion, and let $I$ be an interval in $\mathds R$. Consider a function $F : I \times Z \rightarrow W$. Let $w$ be an arbitrary element of $W$ and define the associated Nemytskii operator $G:I\times \mX\rightarrow W$ by
\be\label{eqdefnemytskiiop}
G(t,u)= \left\{\begin{array}{ll}
                  F(t,u(t)), &\mbox{ if } u(t)\in Z,\\
                  w, &\mbox{otherwise},
                 \end{array}\right.
\ee
where $\mX=\Cloc(I,X)$. Assume that every Borel subset of $Z$ is a Borel subset of $X$ and that $F$ is a $(\fL_I\otimes\mathcal\fB_Z,\fB_W)$-measurable function. Then $G$ is $(\fL_I\otimes\fB_{\mX},\fB_W)$-measurable.
\end{prop}

\section{Abstract Results}\label{secabstractresults}

In this section, we present our abstract framework for the theory of statistical solutions. First, in Section \ref{subsectypesofss}, we give our general definitions of statistical solutions in trajectory space and in phase space. Then, in Sections \ref{subsecexistenceoftrajss} and \ref{subsecexistenceofss} we prove the main results on the existence of these general types of statistical solutions with respect to a given initial data.

\subsection{Types of Statistical Solutions}
\label{subsectypesofss}

We first define statistical solutions in the space of continuous paths $\mX=\Cloc(I,X)$, in a Hausdorff space $X$. They are named \textbf{trajectory statistical solutions}, owing to the fact that they are measures carried by a measurable subset of a certain set $\mU$ in $\mX$ which, in applications, would consist in the set of trajectories, i.e. the set of solutions, in an appropriate sense, of a given evolution equation. At this abstract level, however, there is no evolution equation, and the problem is simply to find a measure carried by a given subset $\mU$ of $\mX$. As such, this is a trivial problem, as showed in Remark \ref{trivialdiracstatsol}. The interesting and difficult problem is the corresponding Initial Value Problem  \ref{ivp_tss}. Nevertheless, we start with the following definition.

\begin{defs}\label{def-stat-sol}
Let $X$ be a Hausdorff space and let $I \subset \mathds R$ be an arbitrary interval. Consider $\mX=\Cloc(I,X)$ and let $\mU$ be a subset of $\mX$. We say that a Borel probability measure $\rho$ on $\mX$ is a \textbf{$\mU$-trajectory statistical solution over $I$} (or simply a \textbf{trajectory statistical solution}) if
\renewcommand{\theenumi}{\roman{enumi}}
\begin{enumerate}
\item $\rho$ is tight;
\item $\rho$ is carried by a Borel subset of $\mX$ included in $\mU$, i.e., there exists $\mV\in\fB_\mX$ such that $\mV\subset \mU$ and $\rho(\mX \setminus \mV)=0$.
\end{enumerate}
\end{defs}

\begin{rmk}
Our abstract definition of a trajectory statistical solution was inspired by the concept of a Vishik-Fursikov measure given in \cite{FRT2013}. Such measures are defined within the context of the Navier-Stokes equations and have the property of being carried by their set of weak solutions, called Leray-Hopf weak solutions. In \cite[Propositions 2.9 and 2.12]{FRT2013} it is proved that the set of Leray-Hopf weak solutions is a Borel set in the corresponding space of continuous paths. However, since we do not know whether this is always the case in every application, we prefer not to assume that $\mU$ is Borel, and assume instead that there exists a Borel subset of $\mU$ that carries the measure $\rho$.
\end{rmk}
\begin{rmk}
\label{urhobarmeasurable}
From the Definition \ref{def-stat-sol}, however, we see that $\mU\setminus \mV$ belongs to the Borel null set $\mX\setminus \mV$, hence we can certainly say that $\mU$ is measurable with respect to the Lebesgue completion of $\rho$, which we denote by $\bar\rho$, and so that $\bar\rho$ is carried by $\mU$.
\end{rmk}

Now we define statistical solutions in phase space. Unlike the definition of a trajectory statistical solution, which is given by a single measure defined on $\mX$, this second type of statistical solution consists in a family of measures defined on the Hausdorff space $X$ and parametrized by an index $t$ varying in an interval $I \subset \mathds R$. For the statistical solutions in phase space, we do need an evolution equation, although minimal hypotheses will be needed, as described in Section \ref{subsecevoleq}. We then have the following definition of statistical solution in phase space:

\begin{defs}\label{def-statsolphasespace}
Let $X$ and $Z$ be Hausdorff spaces and $Y$ be a topological vector space such that $Z\subset X\subset Y_{\rw^*}'$, with continuous injections, where $Y'$ denotes the dual space of $Y$ and $Y_{\rw^*}'$ is the dual space $Y'$ endowed with the weak-star topology. Let $I\subset \mathds{R}$ be an arbitrary interval and consider a function $F: I \times Z \rightarrow Y'$. We say that a family $\{\rho_t\}_{t\in I}$ of Borel probability measures in $X$ is a \textbf{statistical solution in phase space} (or simply a \textbf{statistical solution}) of the evolution equation $u_t = F(t,u)$, over the interval $I$, if the following conditions are satisfied:
\renewcommand{\theenumi}{\roman{enumi}}
\begin{enumerate}
\item\label{ssphasespace1} The function
\[t\mapsto \int_X \varphi(u)\;\rd\rho_t(u)\]
is continuous on $I$, for every $\varphi \in \mC_\rb(X)$.
\item \label{ssphasespace2} For almost every $t\in I$, the measure $\rho_t$ is carried by $Z$ and the function $u \mapsto \langle F(t,u),v\rangle_{Y', Y}$ is $\rho_t$-integrable, for every $v\in Y$. Moreover, the map
\[t\mapsto \int_X \langle F(t,u),v\rangle_{Y',Y}\;\rd\rho_t(u)\]
belongs to $L^1_{\rloc}(I)$, for every $v\in Y$.
\item \label{ssphasespace3} For any cylindrical test function $\Phi$ in $Y'$, it follows that
\be\label{meaneq}
\int_X\Phi(u)\;\rd\rho_t(u)=\int_X\Phi(u)\;\rd\rho_{t'}(u)+\int_{t'} ^t\int_X\langle F(s,u),\Phi'(u)\rangle_{Y',Y}\;\rd\rho_s(u)\;\rd s,
\ee
for all $t,t'\in I$.
\end{enumerate}
\end{defs}

In Definition \ref{def-statsolphasespace}, equation \eqref{meaneq} represents a Liouville-type equation similar to that from statistical mechanics. In order to motivate the definition \eqref{meaneq}, let us suppose that a particular statistical solution $\{\rho_t\}_{t\in I}$ is given in the form of a convex combination of Dirac measures,
\[
\rho_t = \frac{1}{N} \sum_{n=1}^N \delta_{u_n(t)},\quad t \in I,
\]
with equal probability $1/N$, where $N\in \mathbb{N}$, and each $u_n$ is a smooth solution (say continuously differentiable as a function from the time interval $I$ into the space $Y'$) of the system
\[ u_t(t) = F(t,u(t)),\quad  t \in I.
\]
We then formally have
\begin{multline}
\frac{\rd}{\rd s} \int_X \Phi(u) \;\rd\rho_s(u) = \frac{\rd}{\rd s} \frac{1}{N} \sum_{n=1}^N \Phi(u_n(s))
= \frac{1}{N} \sum_{n=1}^N \frac{\rd}{\rd s} \Phi(u_n(s)) \\
= \frac{1}{N} \sum_{n=1}^N \left\langle \frac{\rd}{\rd s} u_n(s), \Phi'(u_n(s)) \right\rangle_{Y',Y}
= \frac{1}{N} \sum_{n=1}^N \langle F(s,u_n(s)), \Phi'(u_n(s)) \rangle_{Y',Y} \\
= \int_X \langle F(s,u), \Phi'(u) \rangle_{Y',Y} \;\rd \rho_s(u).
\end{multline}
Thus, integrating with respect to $s$ on $[t',t]$ yields \eqref{meaneq}.

In Section \ref{subsecexistenceofss} (Theorem \ref{thm_projectedss_is_ss}), we prove that the family of measures obtained as the projections of a trajectory statistical solution at each $t \in I$ on the space $X$ is a statistical solution. This shows that every trajectory statistical solution on $\mX$ yields a statistical solution on $X$. However, the converse is not necessarily true.  We then call a statistical solution for which the converse is valid, i.e., which can be written as the projections on $X$ of a trajectory statistical solution, a \textbf{projected statistical solution}, as given below.

\begin{defs}\label{defprojectstatsol}
Let $X$ and $Z$ be Hausdorff spaces and $Y$ be a topological vector space such that $Z\subset X\subset Y_{\rw^*}'$, with continuous injections, where $Y'$ denotes the dual space of $Y$ and $Y_{\rw^*}'$ is the dual space $Y'$ endowed with the weak-star topology. Let $I\subset \mathds{R}$ be an arbitrary interval and let $\mU$ be a subset of $\mX$. Consider a function $F: I \times Z \rightarrow Y'$. We say that a family $\{\rho_t\}_{t\in I}$ of Borel probability measures on $X$ is a \textbf{statistical solution projected from a $\mU$-trajectory statistical solution} of the evolution equation $u_t = F(t,u)$, over the interval $I$ (or simply a \textbf{projected statistical solution}), when $\{\rho_t\}_{t\in I}$ is a statistical solution of the evolution equation $u_t=F(t,u)$ in the sense of Definition \ref{def-statsolphasespace} and there exists a $\mU$-trajectory statistical solution $\rho$ such that $\rho_t=\Pi_t\rho$, for every $t\in I$.
\end{defs}

\begin{rmk}
  \label{newconditioniiissufficient}
  In Definition \ref{def-statsolphasespace}, condition \eqref{ssphasespace3}, since $\Phi'(u)$ is a linear combination of vectors in $Y$ (see \eqref{eqfrechetderphi}), the condition \eqref{ssphasespace2} is sufficient to guarantee that the second integral on the right hand side of \eqref{meaneq} is well-defined. More details in the proof of Theorem \ref{thm_projectedss_is_ss}.
\end{rmk}

\begin{rmk}
  \label{ZequalsYrmk}
  In most applications, the spaces $Z$ and $Y$ in the Definition \ref{def-statsolphasespace} of phase-space statistical solution are taken to be the same. Nevertheless, this is not a requirement in the abstract setting and we allow for this flexibility. This situation is illustrated in Section \ref{subsecreacdiffeq}.
\end{rmk}

\begin{rmk}
  \label{trivialdiracstatsol}
Note that whenever $\mU$ is a nonempty set, we can always obtain a trajectory statistical solution by considering the Dirac measure $\delta_{u}$, for any element $u\in\mU$ ($\delta_u$ is tight and $\{u\}$ is a Borel set in $\mU$ satisfying $\delta_{u}(\{u\})=1$). A statistical solution can then also be easily obtained by considering the family of projections $\{\delta_{u(t)}\}_{t\in I}$. However, our main concern is not simply the existence of a measure or a family of measures satisfying the properties described in Definitions \ref{def-stat-sol} or \ref{def-statsolphasespace}, respectively. Our aim is to prove the existence of such solutions for an initial value problem, as described below.
\end{rmk}

In the case of trajectory statistical solutions, the initial value problem takes the following form:

\begin{prob}[Initial Value Problem for Trajectory Statistical Solutions]\label{ivp_tss}
Let $I\subset \mathds{R}$ be an interval closed and bounded on the left, with left end point $t_0$, and let $X$ be a Hausdorff space. Let $\mX=\Cloc(I,X)$ be the space of continuous paths in $X$ endowed with the compact-open topology. Let $\mU$ be a given subset of $\mX$. Given an ``initial'' tight Borel probability measure $\mu_0$ on $X$, we look for a $\mU$-trajectory statistical solution $\rho$ on $\mX$ satisfying the initial condition $\Pi_{t_0}\rho = \mu_0$, i.e., we look for a measure $\rho\in\mP(\mX)$ satisfying conditions $(i)$ and $(ii)$ of Definition \ref{def-stat-sol} and such that
  \[\rho(\Pi_{t_0}^{-1}(A))=\mu_0(A)\,,\,\,\forall A\in\fB_X.\]
\end{prob}

The corresponding problem for statistical solutions is stated analogously:

\begin{prob}[Initial Value Problem for Statistical Solutions]\label{ivp_ss}
Let $X$ and $Z$ be Hausdorff spaces and $Y$ be a topological vector space such that $Z\subset X\subset Y_{\rw^*}'$, with continuous injections. Let $I\subset \mathds{R}$ be an interval closed and bounded on the left, with left end point $t_0$. Consider $F: I \times Z \rightarrow Y'$. Given an ``initial'' tight Borel probability measure $\mu_0$ on $X$, we look for a statistical solution $\{\rho_t\}_{t\in I}$ of the equation $u_t = F(t,u)$, over the interval $I$, and satisfying the initial condition $\rho_{t_0}=\mu_0$.
\end{prob}

Note that in the Initial Value Problems \ref{ivp_tss} and \ref{ivp_ss} the interval $I$ is considered as being closed and bounded on the left, with left end point $t_0$. This point $t_0$ represents, of course, the initial time in an application, so that $\mu_0$ is the initial measure.

\subsection{Existence of Trajectory Statistical Solutions}
\label{subsecexistenceoftrajss}

In order to obtain the existence of trajectory statistical solutions in the sense of Definition \ref{def-stat-sol} and satisfying a given initial data (Problem \ref{ivp_tss}), we impose some natural assumptions on the subset $\mU\subset \mX$. In fact, our main existence result for the initial value problem for trajectory statistical solutions reads as follows.

\begin{thm}\label{existencestatsol}
Let $X$ be a Hausdorff space and let $I$ be a real interval closed and bounded on the left with left end point $t_0$. Let $\mU\subset \mX$ be a subset of $\mX=\Cloc(I,X)$. Suppose that
\renewcommand{\theenumi}{{H}\arabic{enumi}}
\begin{enumerate}
  \item \label{eqH1b} $\Pi_{t_0}\mU=X$.
\end{enumerate}
Suppose, moreover, that there exists a family $\fK'(X)$ of compact subsets of $X$ such that
\begin{enumerate}
  \setcounter{enumi}{1}
  \item \label{eqH2bii} Every tight Borel probability measure $\mu_0$ on $X$ is inner regular with respect to the family $\fK'(X)$ in the sense of \eqref{defmureg};
  \item \label{eqH2biii} For every $K\in\fK'(X)$, the subset $\Pi_{t_0}^{-1}K\cap\mU$ is compact in $\mX$.
\end{enumerate}
Then, for any tight Borel probability measure $\mu_0$ on $X$ there exists a $\mU$-trajectory statistical solution $\rho$ on $I$ such that $\Pi_{t_0}\rho=\mu_0$.
\end{thm}

Before going into the proof of Theorem \ref{existencestatsol}, let us first discuss the hypotheses of the theorem and the main ideas in its proof. 

In the applications to a given evolution equation, the Hausdorff space $X$ plays the role of the phase space associated with the equation, and the set $\mU$ is the set of solutions in a given sense. These solutions are assumed to be continuous functions defined on a real time-interval $I$ and with values in the phase space $X$, so that $\mU$ is a subset of $\mX = \Cloc(I,X)$. 

Then, hypothesis \eqref{eqH1b} is simply a statement about global existence of solutions for every initial condition in $X$. Indeed, it requires that for every initial condition $u_0 \in X$, there is a solution $u \in \mU$, existing for the whole time interval $I$, and satisfying $u(t_0) = u_0$.

Hypotheses \eqref{eqH2bii} and \eqref{eqH2biii} come together and are essentially a compactness condition on the subsets of solutions with initial conditions in sets belonging to a certain family of compact subsets of the phase space which is sufficiently large to approximate from below every tight Borel probability measure on $X$. In many applications, $\fK'(X)$ may be considered as the entire family of compact sets of $X$, so that hypothesis \eqref{eqH2bii} holds trivially, and then hypothesis \eqref{eqH2biii} follows from usual a~priori estimates and some compactness embedding theorem. This is the case, for instance, for the reaction-diffusion equation presented in Section \ref{subsecreacdiffeq}. 

In other applications, however, such as to the incompressible Navier-Stokes equations and the nonlinear wave equation (Sections \ref{subsecnse} and \ref{subsecnonlinearwaveeq}), the compactness of $\Pi_{t_0}^{-1}K \cap \mU$ is not known to hold for every compact set $K$ in $X$. The reason is related to the fact that the solutions are taken to be weakly continuous in time while being strongly continuous at the initial time (with the phase space $X$ being equal to a separable Banach space, in a set-theoretic sense, but endowed with the corresponding weak topology). It is precisely this strong continuity at the initial time which might be lost for the limit points of $\Pi_{t_0}^{-1}K \cap \mU$. In order to overcome this problem, the family $\fK'(X)$ is taken to be slightly smaller, made only of the strongly compact sets. Using the energy inequality, one proves that the strong compactness of $K$ implies that the limit points of $\Pi_{t_0}^{-1}K \cap \mU$ are also strongly continuous at the initial time, so that $\Pi_{t_0}^{-1}K \cap \mU$ is, in fact, compact in $\mX$. On the other hand, since in a separable Banach space the Borel $\sigma$-algebras corresponding to the strong and the weak topologies coincide (see Section \ref{subsecinjectborel}), any Borel probability measure on $X$ (endowed with the weak topology) is also a Borel measure with respect to the strong topology. Moreover, since every separable Banach space is also a Polish space, any Borel probability measure with respect to the strong topology of X is inner regular with respect to the family of (strongly) compact subsets (see Section \ref{subsecelementsofmeastheory}), showing that hypothesis \eqref{eqH2bii} also holds.

Let us now outline the main ideas of the proof itself. Starting with an initial measure $\mu_0$ in $\PXt$, at a given time $t_0$, our intention is to show the existence of a measure $\rho$ which is a trajectory statistical solution satisfying the initial condition $\Pi_{t_0}\rho = \mu_0$. As usual, this measure $\rho$ is obtained from the limit of a convergent net of measures.

We first consider the case when the initial measure $\mu_0$ is carried by a set $K$ in the family $\fK'(X)$. Since $K$ is a compact set in $X$, then by using the Krein-Milman Theorem we obtain a net $\{\mu_0^\alpha\}_\alpha$ of discrete measures converging to $\mu_0$ in $X$. Using hypothesis \eqref{eqH1b}, we can easily extend each discrete initial measure $\mu_0^\alpha$ to a discrete measure $\rho_\alpha$ in $\mX$; we just apply \eqref{eqH1b} to each point in the support of $\mu_0^\alpha$. By construction, each $\rho_\alpha$ is a tight measure carried by $\KU$, which by hypothesis \eqref{eqH2biii} is a compact set. This implies that $\{\rho_\alpha\}_\alpha$ is a uniformly tight net and then Theorem \ref{thmMXtcompactness} is applied to obtain a subnet converging to some tight measure $\rho$, also carried by $\KU$, which is, in particular, a Borel set in $\mX$. Thus $\rho$ is a trajectory statistical solution, in the sense of Definition \ref{def-stat-sol}. The fact that $\rho$ satisfies the initial condition, i.e., $\Pi_{t_0}\rho=\mu_0$, follows easily from the uniqueness of the limits in $\mM(X,\text{tight})$, guaranteed by Theorem \ref{thmMXthausdorff}.

The proof of the case when $\mu_0$ is not carried by any set $K \in \fK'(X)$, can be reduced to the previous case by using the hypothesis that $\mu_0$ is a tight measure and thus, in particular, inner regular with respect to the family $\fK'(X)$ by hypothesis \eqref{eqH2bii}. The idea, then, consists in decomposing $\mu_0$ as a sum of Borel measures, each being carried by a set in $\fK'(X)$. The previous case can be applied to each of these measures (normalized to probability measures), yielding a countable family of $\mU$-trajectory statistical solutions. Our desired measure is therefore obtained as a (weighted) sum of these particular measures.

In the previous discussion, we skipped some technical details concerning the restriction of the approximating measures to convenient compact subsets. If we assumed that our underlying phase space was completely regular, the proof could be made a bit simpler, as these restrictions would no longer be necessary since in completely regular Hausdorff spaces the weak-star semi-continuity topology coincides with the weak-star topology (see Lemma \ref{lemportmanteau}). But again, looking for a higher degree of generality, we assume only that our phase space $X$ is a Hausdorff space.

\begin{proof}[Proof of Theorem \ref{existencestatsol}]
Let us first suppose that $\mu_0$ is carried by a compact set $K$ in $\fK'(X)$. Using the Krein-Milman Theorem (\cite[Theorem 3.23]{RudinFA}) we obtain a net $\{\mu_0^\alpha\}_\alpha$ of discrete measures in $\mP(K)$ such that $\mu_0^\alpha\wconv \mu_0|_K$. Since each $\mu_0^\alpha$ is a discrete measure, there exist $J_\alpha\in\mathds{N}$, $0< \theta_j^\alpha\leq 1$, and $u_{0,j}^\alpha\in K$, such that
\[\mu_0^\alpha=\sum_{j=1}^{J_\alpha}\theta_j^\alpha\delta_{u_{0,j}^\alpha},\]
with $\sum_{j=1}^{J_\alpha} \theta_j^\alpha = 1$, for every $\alpha$.

From the hypothesis \eqref{eqH1b} it follows that, for each $u_{0,j}^\alpha$, there exists $u_j^\alpha\in\mathcal{U}$ such that $\Pi_{t_0}u_j^\alpha=u_{0,j}^\alpha$. Consider the measure $\rho_\alpha$ defined on $\mX$ by
\[\rho_\alpha=\sum_{j=1}^{J_\alpha}\theta_j^\alpha\delta_{u_j^\alpha}.\]

Note that $\rho_\alpha$ belongs to $\mP(\mX,\text{tight})$ and is carried by $\KU$, which is a compact set by the hypothesis \eqref{eqH2biii}. Thus, $\{\rho_\alpha\}_\alpha$ is clearly a uniformly tight net. By Theorem \ref{thmMXtcompactness}, there is a measure $\rho$ in $\mP(\mX,\text{tight})$ such that, by passing to a subnet if necessary,
\be\label{eq6}
\rho_\alpha \wconv \rho\mbox{ in }\mX.
\ee
Moreover, using Lemma \ref{lemportmanteau}, we find that $\rho$ is carried by $\KU$.

Consider a bounded and upper semicontinuous function $\varphi: X \rightarrow \mathds R$. Then $\varphi \circ \Pi_{t_0}$ is a bounded function on $\mX$. Moreover, since $\Pi_{t_0}$ is continuous, $\varphi \circ \Pi_{t_0}$ is also an upper semicontinuous function on $\mX$. Applying a change of variables as in \eqref{eq00} and using the convergence $\rho_\alpha\wconv\rho$ together with Lemma \ref{lemportmanteau}, we obtain that

\begin{multline*}
\limsup_{\alpha} \int_X \varphi(u) \;\rd \Pi_{t_0} \rho_{\alpha}(u) = \limsup_{\alpha} \int_{\mX} (\varphi \circ \Pi_{t_0})(u) \;\rd \rho_{\alpha}(u) \\
    \leq  \int_{\mX} (\varphi \circ \Pi_{t_0})(u) \;\rd \rho(u) = \int_X \varphi(u) \;\rd \Pi_{t_0} \rho(u).
\end{multline*}
This means, using Lemma \ref{lemportmanteau} once again, that $\Pi_{t_0}\rho_\alpha\wconv\Pi_{t_0}\rho$ in $X$. Furthermore, taking the restrictions of these measures to the compact $K$, we also have that $\Pi_{t_0}\rho_\alpha|_K \wconv \Pi_{t_0}\rho|_K$. On the other hand, we have by construction that
\[\Pi_{t_0}\rho_\alpha|_K=\mu_0^\alpha\wconv\mu_0|_K.\]
Adding this to the fact that $\Pi_{t_0}\rho|_K,\mu_0|_K\in\mP(K,\text{tight})$, we obtain, by the uniqueness of the limit guaranteed by Theorem \ref{thmMXthausdorff}, that $\Pi_{t_0}\rho|_K=\mu_0|_K$. Then, since $\Pi_{t_0}\rho$ and $\mu_0$ are carried by $K$ we find that $\Pi_{t_0}\rho = \mu_0$.

We have thus obtained a measure $\rho \in \mP(\mX,\text{tight})$ carried by the compact and hence Borel set $\Pi_{t_0}^{-1}K\cap\mU\subset\mU$, with the initial condition $\Pi_{t_0}\rho = \mu_0$, which means that we have proved the existence of a trajectory statistical solution satisfying the initial condition in the case that $\mu_0$ is carried by a set $K\in\fK'(X)$.

Now let us consider the case when $\mu_0$ is not carried by any set $K\in\fK'(X)$. Since $\mu_0$ is a tight Borel probability measure on $X$, we have, by the hypothesis \eqref{eqH2bii}, that $\mu_0$ is inner regular with respect to the family $\fK'(X)$. Thus, there exists a sequence $\{K_n\}_n$ of sets in $\fK'(X)$ such that
\begin{equation}\label{eq3}
\mu_0(K_{n+1})>\mu_0(K_n)>0, \quad \text{and} \quad \mu_0(X\backslash K_n)<\frac{1}{n}\,,\,\,\forall n\in\mathds{N}.
\end{equation}
Moreover, we may assume that $K_n\subset K_{n+1}$, for all $n\in\mathds{N}$.

Let $D_1=K_1$ and $D_n=K_n\backslash K_{n-1}$, for every $n\geq 2$. Note that
\[\mu_0\left(X\backslash\bigcup_j D_j\right)=\mu_0\left(X\backslash\bigcup_j K_j\right)\leq\mu_0(X\backslash K_n)<\frac{1}{n},\]
for all $n\in\mathds{N}$. Thus, taking the limit as $n\rightarrow\infty$ above, we obtain that $\mu_0$ is carried by $\bigcup_j D_j$. Then, for every $A\in\fB_X$, since the sets $D_j$, $j\in \mathds{N}$, are pairwise disjoint, we have
\[\mu_0(A)=\mu_0\left(A\cap\left(\bigcup_j D_j\right)\right)=\sum_{j=1}^\infty\mu_0(A\cap D_j).\]

So we may decompose $\mu_0$ as
\[\mu_0=\sum_j\mu_0(D_j)\mu_0^j,\]
where $\mu_0^j$ is the Borel probability measure defined as
\[\mu_0^j(A)=\frac{\mu_0(A\cap D_j)}{\mu_0(D_j)}\,,\,\,\forall A\in\fB_X.\]
Note that each $\mu_0^j$ is well-defined, since $\mu_0(D_1)=\mu_0(K_1)>0$ and
\[\mu_0(D_j)=\mu_0(K_j)-\mu_0(K_{j-1})>0\,,\,\,\forall j\geq 2.\]

Also, since each $\mu_0^j$ is carried by the set $K_j\in\fK'(X)$, using the first part of the proof, for each $j\in\mathds{N}$ we obtain a tight Borel probability measure $\rho_j$ carried by $\Pi_{t_0}^{-1}K_j\cap\mathcal{U}$ and such that $\Pi_{t_0}\rho_j=\mu_0^j$. Let then $\rho$ be the Borel probability measure defined by
\[\rho=\sum_j\mu_0(D_j)\rho_j.\]

Observe that
\[\rho\left(\bigcup_l \Pi_{t_0}^{-1}K_l\cap\mathcal{U}\right)=\sum_j\mu_0(D_j)\rho_j(\Pi_{t_0}^{-1}K_j\cap\mathcal{U})=\sum_j\mu_0(D_j)=1,\]
where the first and second equalities follow from the fact that $\rho_j$ is carried by $\Pi_{t_0}^{-1}K_j\cap\mU$. Thus, $\rho$ is carried by $\bigcup_j \Pi_{t_0}^{-1}K_j\cap\mathcal{U}$, which is a Borel set in $\mX$ and is contained in $\mU$. The fact that $\Pi_{t_0}\rho=\mu_0$ is also easily verified.

It only remains to show that $\rho$ is a tight measure. In order to prove so, consider a Borel set $\mA\in\fB_\mX$ and $\varepsilon>0$. Let $n\in\mathds{N}$ be such that $1/n < \varepsilon/2$. Since $\rho_j$ is a tight measure, there exists, for each $1\leq j\leq n$, a compact set $\mK_j^n\subset \mA$ such that
\[\rho_j(\mA\backslash \mK_j^n)<\frac{\varepsilon}{2n}.\]
Let $\mK^n=\bigcup_{1\leq j\leq n}\mK_j^n$, and note that
\begin{eqnarray*}
\rho(\mA\backslash \mK^n)&=&\sum_{j=1}^\infty\mu_0(D_j)\rho_j(\mA\backslash \mK^n)\\
&\leq&\sum_{j=1}^n\rho_j(\mA\backslash \mK^n)+\sum_{j=n+1}^\infty\mu_0(D_j)\\
&<&\frac{\varepsilon}{2}+\mu_0(X\backslash K_n).
\end{eqnarray*}
Thus, according to \eqref{eq3} and the choice of $n$, it follows that $\rho(\mA\backslash \mK^n)<\varepsilon$. Since each $\mK^n$ is a compact set in $\mX$, this proves that $\rho$ is tight.
\end{proof}

\begin{rmk}
Notice that given an initial tight Borel probability measure $\mu_0$ on $X$, if $\mu_0$ is carried by a set $K \in \fK'(X)$, then the trajectory statistical solution $\rho$ with $\Pi_{t_0}\rho = \mu_0$ obtained in the proof of Theorem \ref{existencestatsol} is carried by the Borel set $\Pi_{t_0}^{-1} K \cap \mU$. On the other hand, if $\mu_0$ is not carried by any set $K\in\fK'(X)$, then given any sequence of sets $K_n$ in $\fK'(X)$, $n\in \mathds{N}$, such that $\mu_0(X\setminus K_n) \rightarrow 0$, as $n\rightarrow \infty$, a trajectory statistical solution $\rho$ with $\Pi_{t_0}\rho = \mu_0$ can be constructed such that it is carried by the Borel set $\mU \cap (\bigcup_n \Pi_{t_0}^{-1} K_n)$.
\end{rmk}

The hypotheses in Theorem \ref{existencestatsol} are needed for the solvability of the initial value problem for trajectory statistical solutions for arbitrary initial tight Borel probability measures on $X$. It may happen, however, that the conditions are met only for a subset of tight Borel probability measures on $X$, or, similarly, that the continuity of the solutions only hold in a coarser topology than that in the phase space. In this regard, the same argument of the proof of Theorem \ref{existencestatsol} yields the following result, which we state without further details.

\begin{thm}\label{existencestatsolweaker}
Let $X$ be a Hausdorff space and let $I$ be a real interval closed and bounded on the left with left end point $t_0$. Let $\mU\subset \mX$ be a subset of $\mX=\Cloc(I,X)$ and let $X_0$ be a Borel subset of $X$. Suppose that
\renewcommand{\theenumi}{{H}\arabic{enumi}{'}}
\begin{enumerate}
  \item \label{eqH1bwk} $\Pi_{t_0}\mU \supset X_0$.
\end{enumerate}
Suppose, also, that there exists a family $\fK'(X_0)$ of compact subsets of $X_0$ such that
\begin{enumerate}
  \setcounter{enumi}{1}
  \item \label{eqH2biiwk} Every tight Borel probability measure $\mu_0$ on $X$ which is carried by $X_0$ is inner regular with respect to the family $\fK'(X_0)$ in the sense of \eqref{defmureg};
  \item \label{eqH2biiiwk} For every $K\in\fK'(X_0)$, the subset $\Pi_{t_0}^{-1}K\cap\mU$ is compact in $\mX$.
\end{enumerate}
Then, for any tight Borel probability measure $\mu_0$ on $X$ which is carried by $X_0$, there exists a $\mU$-trajectory statistical solution $\rho$ on $I$ such that $\Pi_{t_0}\rho=\mu_0$.
\end{thm}

\begin{rmk}
  \label{extensionhdoubleprime}
  The extension of the result to the framework described in Theorem \ref{existencestatsolweaker} is motivated by the original framework of Vishik and Fursikov for the Navier-Stokes equations (see e.g. \cite{VF88}), in which, instead of using the weak continuity in $L^2$ of the weak solutions, one uses the continuity in sufficiently large negative powers of the Stokes operator. In this case, one can take, for instance, $X=V'=D(A^{-1/2})$ and $X_0=H_\rw$, where $H$ and $V$ are as in Section \ref{subsecnse}.
\end{rmk}

\subsection{Existence of Statistical Solutions in Phase Space}\label{subsecexistenceofss}

We start with the following lemma.

\begin{lem}\label{lemprojsscarriedbyY}
Let $X$ be a Hausdorff space, $I \subset \mathds R$ be an interval, and $\mX=\Cloc(I,X)$. Let $\rho$ be a Borel probability measure on $\mX$. Suppose that $Z$ is a Borel subset of $X$ and that $\rho$ is carried by a Borel subset $\mV$ of the space $\mZ$ defined in \eqref{defZ}. Then, the projected measure $\rho_t = \Pi_t\rho$ is carried by $Z$, for almost every $t\in I$.
\end{lem}

\begin{proof}
Consider the characteristic function $\chi_Z$ of $Z$ in $X$. Since $\mV$ is such that, for all $u\in \mV$, we have $u(t)\in Z$, for almost every $t\in I$, then $\chi_Z(u(t))=1$, for almost every $t \in I$. Thus,
\be\label{lemrhotcarriedbyY2}
 \int_\mV \int_{t_1}^{t_2} \chi_Z(u(t)) \;\rd t \,\rd \rho(u) = \int_\mV \int_{t_1}^{t_2} 1 \;\rd t \,\rd \rho(u) = (t_2-t_1)\rho(\mV),
\ee
for every $t_1, t_2\in I$, $t_1<t_2$. Notice that the map $(t,u) \mapsto \chi_Z(u(t))$ is the composition of the function $\chi_Z$ with the evaluation operator $U(t,u) = u(t)$ given in \eqref{eqdefUonIX}. Since $Z$ is a Borel set in $X$, then $\chi_Z$ is a Borel function on $X$. Moreover, since the evaluation operator $U$ is continuous and hence Borel from $I\times \mX$ into $X$, it follows that the composition map $(t,u)\mapsto \chi_Z(u(t))$ is also a real-valued Borel function on $I\times \mX$.  Then, we apply Tonelli's Theorem (\cite[Theorem III.11.14]{dunfordschwartz}) to the left-hand side of \eqref{lemrhotcarriedbyY2} to obtain that
\[
\int_{t_1}^{t_2} \int_{\mV} \chi_Z(u(t)) \;\rd \rho(u) \, \rd t = (t_2-t_1)\rho(\mV),
\]
for all $t_1,t_2 \in I$, $t_1<t_2$. Using that $\rho$ is a probability measure carried by $\mV$, we find that
\[ \int_{\mV} \chi_Z(u(t)) \;\rd \rho(u) = 1,
\]
for almost every $t \in I$. Using again that $\rho$ is carried by $\mV$, we see that
\[
\rho_t(Z) = \int_X \chi_Z(v)\rho_t(v) =
\int_{\mX} \chi_Z(u(t)) \rho(u) = \int_{\mV} \chi_Z(u(t)) \;\rd \rho(u) = 1,
\]
for almost every $t\in I$. Therefore, since $\rho_t$ is a probability measure, $\rho_t(X\setminus Z)=0$, and hence $\rho_t$ is carried by $Z$, for almost every $t\in I$.
\end{proof}

Now we prove that the family of measures obtained as the projections of a trajectory statistical solution at each $t \in I$ on $X$ is a statistical solution, in the sense of Definition \ref{def-statsolphasespace}. 

\begin{thm}\label{thm_projectedss_is_ss}
Let $X$ and $Z$ be Hausdorff spaces and $Y$ be a topological vector space such that $Z\subset X\subset Y_{\rw^*}'$, with continuous injections, where $Y'$ denotes the dual space of $Y$ and $Y_{\rw^*}'$ is the dual space $Y'$ endowed with the weak-star topology. Consider an interval $I \subset \mathds R$ and a subset $\mU \subset \mX$, where $\mX=\Cloc(I,X)$. Let $\rho$ be a $\mU$-trajectory statistical solution and let $\mV$ be a Borel subset of $\mX$ such that $\mV\subset \mU$ and $\rho(\mV)=1$. Assume that $\fB_Z \subset \fB_X$; that $\mU \subset \mX_1$,  where $\mX_1$ is defined in \eqref{defXprime}; and that $F: I \times Z \rightarrow Y'$ is an $(\fL_I\otimes\fB_Z,\fB_{Y'})$-measurable function such that \eqref{Ftuislocint} holds and $u_t = F(t,u)$ in the weak sense \eqref{evoleqweak}. Assume, moreover, that
\be\label{intF}
t\mapsto \int_{\mV} \left|\langle F(t,u(t)),v\rangle_{Y',Y}\right|\;\rd\rho(u) \in L^{1}_{\rloc}(I),
\ee
for every $v\in Y$. Then,
\be\label{liouvilleeq}
\int_\mV\Phi(u(t))d\rho(u)=\int_\mV\Phi(u(t'))d\rho(u)+\int_{t'}^t\int_\mV\langle F(s,u(s)),\Phi'(u(s))\rangle_{Y',Y}d\rho(u)ds,
\ee
for all $t,t'\in I$ and for all cylindrical test function $\Phi$ in $Y'$. Moreover, the function
\be\label{intphi}
t\mapsto \int_{\mV}\varphi(u(t))\;\rd \rho(u)
\ee
is continuous on $I$ for every $\varphi \in \mC_\rb(X)$. In particular, the family of projections $\{\rho_t\}_{t\in I}$, where $\rho_t=\Pi_t\rho$, is a statistical solution in phase space of the equation $u_t=F(t,u)$, over the interval $I$, in the sense of Definition \ref{def-statsolphasespace}.
\end{thm}

\begin{proof}
Since $\mV \subset \mU\subset\mX_1$, it follows by the definition of $\mX_1$ that, for every $u\in \mV$, $u(t)\in Z$, for almost every $t\in I$. Since $\fB_Z\subset\fB_X$, we have, in particular, that $Z$ is a Borel subset of $X$. Therefore, it follows from Lemma \ref{lemprojsscarriedbyY} that $\rho_t=\Pi_t\rho$ is carried by $Z$, for almost every $t\in I$. This means that $\rho$ is carried by $\Pi_t^{-1}Z$, for almost every $t \in I$. Hence, the integrals in \eqref{intF} and \eqref{liouvilleeq} with respect to $\rho$ in $\mV$, with integrands containing the mapping $u\in \mV \mapsto F(t,u(t))$, are well-defined almost everywhere in $I$.

Now consider a cylindrical test function $\Phi: Y' \rightarrow \mathds R$ given by
\[
\Phi(u)=\phi(\langle u, v_1\rangle_{Y',Y}, \ldots, \langle u,
v_k\rangle_{Y',Y}), \quad \forall u \in Y',
\]
where $\phi$ is a continuously-differentiable real-valued function on $\mathds R^k$ with compact support, $k\in \mathbb{N}$, and $v_1,\ldots,v_k\in Y$.

For every $u\in \mU$, since $\mU$ is included in the space $\mY_1$ defined in \eqref{defYprime}, the function $t \mapsto \langle u(t), v_j \rangle_{Y',Y}$ is absolutely continuous on $I$ and, for almost every $t \in I$,
\[\frac{\rd}{\rd t}\langle u(t),v_j\rangle_{Y',Y} =
\langle F(t,u(t)), v_j \rangle_{Y' ,Y}, \quad \forall j=1,\ldots, k.\]
Thus,
\begin{multline}\label{eqderPhi}
\frac{\rd}{\rd t}\Phi(u(t))=\sum_{j=1}^k\partial_j\phi(\langle u(t),v_1\rangle_{Y',Y},
\ldots, \langle u(t),v_k\rangle_{Y',Y})\frac{\rd}{\rd t}\langle u(t),v_j\rangle_{Y',Y}\\
=\sum_{j=1}^k\partial_j\phi(\langle u(t),v_1\rangle_{Y',Y}, \ldots,\langle u(t),v_k\rangle_{Y',Y})\langle F(t,u(t)),v_j\rangle_{Y',Y}\\
=\langle F(t,u(t)),\Phi'(u(t))\rangle_{Y' ,Y},
\end{multline}
where $\Phi'$ is the differential of $\Phi$ in $Y'$, given in \eqref{eqfrechetderphi}.

Let us show that, for every $u\in \mU$, the mapping $t \mapsto \Phi (u(t))$ is absolutely continuous on $I$. Since each $\partial_j \phi$ is bounded in $\mathds R^k$, there exists $M > 0$ such that $\| \nabla \phi (\bx)\| \leq M$, for every $\bx \in \mathds R^k$, where $\| \cdot \|$ denotes the norm in $\mathds R^k$. Then, given any finite sequence of pairwise disjoint subintervals $\{(t_j,s_j)\}_{j=1}^N$ in $I$, we obtain, from the Mean Value Theorem,
\begin{multline*}
\sum_{j=1}^N | \phi( \langle u(s_j),v_1 \rangle_{Y',Y}, \ldots, \langle u(s_j),v_k \rangle_{Y',Y}) - \phi( \langle u(t_j),v_1 \rangle_{Y',Y}, \ldots, \langle u(t_j),v_k \rangle_{Y',Y}) | \\
\leq M \sum_{j=1}^N \| ( \langle u(s_j),v_1 \rangle_{Y',Y}, \ldots, \langle u(s_j),v_k \rangle_{Y',Y}) - ( \langle u(t_j),v_1 \rangle_{Y',Y}, \ldots, \langle u(t_j),v_k \rangle_{Y',Y}) \|.
\end{multline*}
Thus, the absolute continuity of the mapping $t \mapsto \Phi(u(t))$ follows by using that each mapping $t \mapsto \langle u(t), v_j \rangle_{Y',Y}$ is absolutely continuous, for $j = 1, \ldots, k$.

Therefore, from \eqref{eqderPhi} we obtain that
\be\label{eqtfcPhi}
\Phi(u(t))=\Phi(u(t'))+\int_{t'}^t\langle F(s,u(s)),\Phi'(u(s))\rangle_{Y' ,Y}ds,
\ee
for every $t,t'\in I$ and every $u \in \mU$.

Consider the function
\begin{eqnarray*}
H : I \times \mX & \rightarrow & \mathds R \\
    (t,u) & \mapsto & \Phi(u(t)).
\end{eqnarray*}
Denote by $\iota$ the continuous injection of $X$ into $Y'_{w*}$ and let $U : I \times \mX \rightarrow X$ be the function defined in \eqref{eqdefUonIX}. Then $H$ can be written as the composition $\Phi \circ \iota \circ U$. And since $\Phi$, $\iota$ and $U$ are continuous functions, so is $H$. Similarly, we obtain that the mapping $(t,u) \mapsto \Phi'(u(t))$ is continuous on $I \times \mX$.

Using the function $G: I \times \mX \rightarrow Y'$ defined in \eqref{eqdefnemytskiiop}, with $W=Y'$ and $w=0$, we write
\[
\langle G(t,u) , \Phi'(u(t))\rangle_{Y' ,Y}
  = \begin{cases} \langle F(t,u(t)),\Phi'(u(t))\rangle_{Y' ,Y}, & \text{if } u(t)\in Z, \\
0, & \text{if } u(t) \in X\setminus Z.
\end{cases} 
\]
Since $F$ is $(\fL_I\otimes\fB_Z,\fB_{Y'})$-measurable by hypothesis, we have from Proposition \ref{propnemytskii} that $G$ is $(\fL_I\otimes\fB_{\mX},\fB_{Y'})$-measurable. Therefore, using also the continuity of the function $p: Y'\times Y \rightarrow \mathds R$ given by
\[
p(u,v) = \langle u, v \rangle_{Y',Y},
\]
it follows that the mapping $(t,u) \mapsto \langle G(t,u),\Phi'(u(t))\rangle_{Y' ,Y}$ is $(\fL_I\otimes\fB_{\mX})$-measurable. Moreover, recall that $\Phi'(u(t))$ is of the form (see \eqref{eqfrechetderphi})
\[
  \Phi'(u(t)) = \sum_{j=1}^k\partial_j\phi(\langle u, v_1\rangle_{Y',Y}, \ldots, \langle u, v_k\rangle_{Y',Y}) v_j,
\]
where $k\in \mathds N$, $v_1,\ldots, v_k \in Y$, and $\phi$ is a continuously differentiable real-valued function on $\mathds R^k$ with compact support. Thus, 
\[
\langle F(t,u(t)),\Phi'(u(t))\rangle_{Y' ,Y} = \sum_{j=1}^k\partial_j\phi(\langle u, v_1\rangle_{Y',Y}, \ldots, \langle u, v_k\rangle_{Y',Y}) \langle F(t,u(t)),v_j\rangle_{Y' ,Y}.
\]
Since the derivatives $\partial_j\phi$ are uniformly bounded, we have that
\[
\left|\langle F(t,u(t)),\Phi'(u(t))\rangle_{Y' ,Y}\right| \leq \sum_{j=1}^k C_j \left|\langle F(t,u(t)),v_j\rangle_{Y' ,Y}\rangle \right|,
\]
for suitable constants $C_1, \ldots, C_k$. Using hypothesis \eqref{intF} and Tonelli's Theorem on nonnegative measurable functions (\cite[Theorem III.11.14]{dunfordschwartz}), we then find that the map
\[ (t,u) \mapsto \langle F(t,u(t)),\Phi'(u(t))\rangle_{Y' ,Y}
\]
is integrable on $I\times \mX$, with respect to the product of the Lebesgue measure and the measure $\rho$.

Thus, from \eqref{eqtfcPhi}, and considering that $\mV\subset \mU$ is a Borel subset of $\mX$, we obtain
\be\label{eqalmostliouville}
\int_\mV\Phi(u(t))\;\rd\rho(u)=\int_\mV\Phi(u(t'))\;\rd\rho(u)+\int_\mV\int_{
t'} ^t\langle G(s,u),\Phi'(u(s))\rangle_{Y' ,Y}\;\rd s \,d\rho(u),
\ee
for every $t,t'\in I$. Applying Fubini's Theorem (\cite[Theorem III.11.9]{dunfordschwartz}) to the second term on the right hand side of \eqref{eqalmostliouville} we find that
\begin{multline}\label{meanevoleq}
\int_\mV\Phi(u(t))d\rho(u) = \int_\mV\Phi(u(t'))d\rho(u)+\int_{t'}^t\int_\mV\langle G(s,u),\Phi'(u(s))\rangle_{Y' ,Y}d\rho(u)ds \\
= \int_\mV\Phi(u(t'))d\rho(u)+\int_{t'}^t\int_\mV\langle F(s,u(s)),\Phi'(u(s))\rangle_{Y' ,Y}d\rho(u)ds,
\end{multline}
for all $t,t'\in I$. This proves the mean equation \eqref{liouvilleeq}.

Now consider a function $\varphi \in \mC_\rb(X)$ and let us prove that the function defined in \eqref{intphi} is continuous on $I$. Given $\tilde{t}\in I$, consider a sequence $\{t_n\}_n$ in $I$ such that $t_n\rightarrow \tilde{t}$. Since every $u\in \mV$ is continuous from $I$ into $X$, it follows that
\[\varphi(u(t_n))\rightarrow \varphi(u(\tilde{t})), \quad \forall u\in\mV.\]
Since $\varphi$ is in particular a bounded function on $X$, we obtain, from the Lebesgue Dominated Convergence Theorem, that
\[\int_{\mV} \varphi(u(t_n))\;\rd\rho(u) \rightarrow \int_{\mV} \varphi(u(\tilde{t}))\;\rd \rho(u).\]
Thus, the function defined in \eqref{intphi} is continuous on $I$. 

Concerning the family $\{\rho_t\}_{t\in I}$ of the projections $\rho_t = \Pi_t\rho$, $t\in I$, it is immediate to see that condition \eqref{ssphasespace1} of Definition \ref{def-statsolphasespace} follows from the continuity of \eqref{intphi}, that condition \eqref{ssphasespace2} of Definition \ref{def-statsolphasespace} follows from hypothesis \eqref{intF}, and that condition \eqref{ssphasespace3} of Definition \ref{def-statsolphasespace} follows from the mean equation \eqref{liouvilleeq}. Therefore, $\{\rho_t\}_{t\in I}$ is a statistical solution in phase space.
\end{proof}

The next result provides a solution for the Initial Value Problem \ref{ivp_ss}. Given an initial measure $\mu_0$ on $X$, we use Theorem \ref{existencestatsol} to obtain a trajectory statistical solution, which is then projected at each time $t$ to yield a family of measures on $X$. Thanks to Theorem \ref{thm_projectedss_is_ss}, this family of projections is a statistical solution.

\begin{thm}\label{thm-existence-projectedss}
Let $X$ and $Z$ be Hausdorff spaces and $Y$ be a topological vector space such that $Z\subset X\subset Y_{\rw^*}'$, with continuous injections, where $Y'$ denotes the dual space of $Y$ and $Y_{\rw^*}'$ is the dual space $Y'$ endowed with the weak-star topology. Let $I \subset \mathds R$ be an interval closed and bounded on the left with left end point $t_0$ and let $\mU$ be a subset of $\mX=\Cloc(I,X)$. Suppose that $\mU$ satisfies the hypotheses \eqref{eqH1b}, \eqref{eqH2bii}, and \eqref{eqH2biii} of Theorem \ref{existencestatsol}, and assume, moreover, that the following conditions hold
\renewcommand{\theenumi}{{H}\arabic{enumi}}
\begin{enumerate}
  \setcounter{enumi}{3}
  \item \label{Hfbzinfbx} $\fB_Z \subset \fB_X$;
  \item \label{HuteqinmU} $\mU \subset \mX_1$, where $\mX_1$ is defined in \eqref{defXprime}, and $F: I \times Z \rightarrow Y'$ is an $(\fL_I\otimes\fB_Z,\fB_{Y'})$-measurable function such that \eqref{Ftuislocint} holds, and $u_t=F(t,u)$ in the weak sense \eqref{evoleqweak};
  \item \label{HFgamma}  there exists a function $\gamma: I \times X \times Y \rightarrow \mathds R$ such that, for every $v\in Y$, the function $(t,u)\mapsto \gamma(t,u,v)$ is $(\fL_I\otimes \fB_X)$-measurable and
\be\label{Fgamma}
   \int_{t_0}^t\left|\langle F(s,u(s)),v\rangle_{Y',Y}\right| \;\rd s\leq \gamma(t,u(t_0),v), \quad \forall t\in I, \;\forall u\in \mU.
\ee
\end{enumerate}
Then, for any tight Borel probability measure $\mu_0$ in $X$ such that
\be\label{gammaL1}
  \int_X \gamma(t,u_0,v)\; d\mu_0(u_0) < \infty, \quad \mbox{for a.e. } t\in I, \mbox{ and for all } v\in Y,
\ee
there exists a projected statistical solution $\{\rho_t\}_{t\in I}$ of $u_t=F(t,u)$, over the interval $I$, associated with a $\mU$-trajectory statistical solution, such that $\rho_{t_0} = \mu_0$.
\end{thm}

\begin{proof}
Since $\mU$ is a subset of $\mX$ satisfying all the hypotheses of Theorem \ref{existencestatsol} and $\mu_0$ is a tight Borel probability measure on $X$, it follows from that theorem that there exists a $\mU$-trajectory statistical solution $\rho$ on $I$ such that $\Pi_{t_0}\rho=\mu_0$.

Since $\rho$ is a $\mU$-trajectory statistical solution, there exists a subset $\mV$ of $\mU$ such that $\mV\in\fB_{\mX}$ and $\rho(\mV)=1$. As in the proof of Theorem \ref{thm_projectedss_is_ss}, it follows from the hypotheses \eqref{Hfbzinfbx} and \eqref{HuteqinmU} that the integrand on the left hand side of \eqref{Fgamma} is integrable with respect to the product of the Lebesgue measure on $I$ and the measure $\rho$ on $\mX$. Then, we take the integral in \eqref{Fgamma} with respect to $\rho$ and apply Tonelli's Theorem (\cite[Theorem III.11.14]{dunfordschwartz}) to find that
\begin{multline*}
\int_{t_0}^t\int_{\mV}\left| \langle F(s,u(s)), v\rangle_{Y',Y}\right| \;\rd\rho(u)\rd s \leq \int_{\mV}\gamma(t,u(t_0),v) \;\rd\rho(u) \\
  =\int_X \gamma(t,u_0,v)\;\rd\Pi_{t_0}\rho(u_0) = \int_X \gamma(t,u_0,v)\;\rd\mu_0(u_0).
\end{multline*}
Thus from \eqref{gammaL1} we obtain that
\be
t\mapsto \int_{\mV}\left| \langle F(t,u(t)), v\rangle_{Y',Y}\right|\;\rd\rho(u) \in L^{1}_{\rloc}(I).
\ee
This proves condition \eqref{intF}, so that, together with \eqref{HuteqinmU}, $F$ is an $(\fL_I\otimes\fB_Z,\fB_{Y'})$-measurable function satisfying all the hypotheses of Theorem \ref{thm_projectedss_is_ss}. All the other hypotheses on $\mU$ and on $\rho$ of Theorem \ref{thm_projectedss_is_ss} are already assumed to be satisfied. Then, it follows from that theorem that the family $\{\rho_t\}_{t\in I}$, with $\rho_t=\Pi_t\rho$, is a statistical solution satisfying $\rho_{t_0}=\Pi_{t_0}\rho=\mu_0$.
\end{proof}

Similarly, we may consider an extension of Theorem \ref{thm-existence-projectedss} with the set of hypotheses \eqref{eqH1bwk}, \eqref{eqH2biiwk}, and \eqref{eqH2biiiwk} described in Theorem \ref{existencestatsolweaker}. In this case, we have the following result, which we state without further comments.

\begin{thm}
  \label{thmextensionhdoubleprimephasespace}
  Consider the framework of Theorem \ref{thm-existence-projectedss} and let $X_0$ be a Borel subset of $X$. Suppose that $\mU$ satisfies, instead, the conditions \eqref{eqH1bwk}, \eqref{eqH2biiwk}, and \eqref{eqH2biiiwk} described in Theorem \ref{existencestatsolweaker}. Assume, moreover, that the conditions \eqref{Hfbzinfbx}, \eqref{HuteqinmU} and \eqref{HFgamma} of Theorem \ref{thm-existence-projectedss} hold.
Then, for any tight Borel probability measure $\mu_0$ in $X$ which is carried by $X_0$ and satisfies \eqref{gammaL1}, there exists a projected statistical solution $\{\rho_t\}_{t\in I}$, associated with a $\mU$-trajectory statistical solution, such that $\rho_{t_0} = \mu_0$.
\end{thm}

\begin{rmk}
  \label{rmkconditionwithnormyprime}
  Concerning the hypothesis \eqref{HFgamma} and the condition \eqref{gammaL1} of Theorem \ref{thm-existence-projectedss}, we will, in fact, have, in many applications, that $Y$ is a Banach space, and we will obtain the stronger condition that there exists an $(\fL_I\otimes \fB_X)$-measurable function $\gamma: I \times X \rightarrow \mathds R$ such that
\be\label{Fgammanorm}
   \int_{t_0}^t\|F(s,u(s))\|_{Y'} \;\rd s\leq \gamma(t,u(t_0)), \quad \forall t\in I, \;\forall u\in \mU.
\ee
In this case, the existence of a projected statistical solution for the initial value problem will hold for any initial Borel probability measure $\mu_0$ such that
\be\label{gammaL1norm}
  \int_X \gamma(t,u_0)\; d\mu_0(u_0) < \infty, \quad \mbox{for a.e. } t\in I,
\ee
\end{rmk}

In the following result, we obtain a mean energy inequality for trajectory statistical solutions, provided the individual solutions satisfy a corresponding energy inequality.
\begin{prop}\label{prop-energy-ineq}
Let $X, Z, Y$ as in the statement of Theorem \ref{thm_projectedss_is_ss}. 
Consider $\mU\subset \mX_1$, where $\mX_1$ is defined in \eqref{defXprime}. Let $\rho$ be a $\mU$-trajectory statistical
solution and let $\mV$ be a Borel subset of $\mX=\Cloc(I,X)$ such that $\mV\subset \mU$ and $\rho(\mV)=1$.
Suppose that there exist functions $\alpha: I\times X\rightarrow\mathds R$
and  $\beta: I\times Z\rightarrow\mathds R$ which are $(\fL_I\otimes\mathcal\fB_X,\fB_\mathds{R})$ and $(\fL_I\otimes\mathcal\fB_Z,\fB_\mathds{R})$ measurable, respectively, and satisfy the following conditions
\renewcommand{\theenumi}{\roman{enumi}}
\begin{enumerate}
 \item \label{propenergyineqi} $(t,u)\mapsto
\alpha(t,u(t))$ belongs to $L^1(J\times \mV,\lambda \times \rho)$, for every compact subset $J \subset I$;
 \item \label{propenergyineqii} $(t,u)\mapsto
\beta(t,u(t))$ belongs to $L^1(J\times \mV,\lambda \times \rho)$, for every compact subset $J \subset I$;
 \item \label{propenergyineqiii} For $\rho$-almost every $u\in \mV$ it follows that
\begin{equation}\label{energyineqiii}
 \frac{\rd}{\rd t}\alpha(t,u(t))+\beta(t,u(t))\leq 0,
\end{equation}
in the sense of distributions on $I$, i.e.,
\be\label{abstract_energyineq}
 -\int_{I}\varphi'(s)\alpha(s,u(s))ds+\int_{I}\varphi(s)\beta(s,u(s))ds\leq
0,
\ee
for all nonnegative test functions $\varphi\in \mC_\rc^\infty(I)$.
\end{enumerate}
Then,
\begin{equation}
  \label{meanpropenergyineqiii}
 -\int_{I}\int_{\mV}\varphi'(s)\alpha(s,u(s))d\rho(u)
ds+\int_{I}\int_{\mV}\varphi(s)\beta(s,u(s))d\rho(u) ds\leq 0,
\end{equation}
for all nonnegative test functions $\varphi\in \mC_\rc^\infty(I)$. In terms of the projected statistical solution in phase space $\{\rho_t\}_{t\in I}$, where $\rho_t=\Pi_t\rho$, we can write \eqref{meanpropenergyineqiii} as
\begin{equation}
  \label{meanpropenergyineqiiiprojected}
 -\int_{I}\int_{\mV}\varphi'(s)\alpha(s,u)d\rho_s(u)
ds+\int_{I}\int_{\mV}\varphi(s)\beta(s,u)d\rho_s(u) ds\leq 0.
\end{equation}
\end{prop}
\begin{proof}
The proof follows by integrating \eqref{abstract_energyineq} with respect to $\rho$ on $\mV$ and then applying Fubini's Theorem (\cite[Theorem III.11.9]{dunfordschwartz}) by using hypotheses \eqref{propenergyineqi} and \eqref{propenergyineqii}.
\end{proof}

The motivation for considering the passage from the inequality \eqref{energyineqiii}, valid for individual weak solutions, to the mean inequality \eqref{meanpropenergyineqiii} comes from the Navier-Stokes equations (see Section \ref{subsecnse}) and other similar equations from fluid flows. It also appears in different types of equations, such as the nonlinear wave equation considered in Section \ref{subsecnonlinearwaveeq}. In some situations, however, such as in the case of the Reaction-Diffusion equation considered in Section \ref{subsecreacdiffeq}, the individual weak solutions satisfies in fact an energy-type equality, and of course this equality is similarly passed on to the statistical solutions. In the case of equality, the test functions are allowed to assume negative values. We state this result as follows, omitting the proof since it follows along the same lines as that of Proposition \ref{prop-energy-ineq}.
\begin{prop}\label{prop-energy-eq}
  Under the hypotheses of Proposition \ref{prop-energy-ineq}, if the equality holds in \eqref{energyineqiii}, then the equality holds in \eqref{meanpropenergyineqiii} and \eqref{meanpropenergyineqiiiprojected}, for any test function $\varphi\in \mC_\rc^\infty(I)$, as well.
\end{prop}

\section{Applications}
\label{secapplications}

In this section we apply the general framework developed previously to obtain the existence of statistical solutions for particular examples of evolution equations.

\subsection{Incompressible Navier-Stokes Equations}
\label{subsecnse}

The Navier-Stokes equations are a commonly used model in the study of turbulent flows. In their three-dimensional form for incompressible Newtonian fluids, these equations are written as
\be\label{nse1}
\frac{\partial\bu}{\partial t}-\nu\Delta\bu+(\bu\cdot\nabla)\bu+\nabla p=\textbf{f},
\ee
\be\label{nse2}
\nabla\cdot\bu=0,
\ee
where $\bu=(u_1,u_2,u_3)$ is the velocity field, $p$ is the kinematic pressure, $\f$ represents a given body force applied to the fluid, and $\nu$ is the parameter of kinematic viscosity. We consider $\bu$, $p$ and $\f$ as functions of a space variable $\bx$ and a time variable $t$, with $\bx$ varying in a set $\Omega \subset \mathds R^3$ and $t$ varying in an interval $I \subset \mathds R$.

For a physical formulation of the equations we refer the reader to the books by Landau and Lifshitz \cite{Landau} and Batchelor \cite{Batchelor}. For the mathematical approach, see Ladyzhenskaya \cite{Lady}, Temam \cite{temam84,temam1995} and Constantin and Foias \cite{bookcf1988}.

Our main concern in this section is to apply the abstract theory developed in the previous sections to prove the existence of statistical solutions to the Navier-Stokes equations \eqref{nse1}-\eqref{nse2}. For more specific discussions on the notion of statistical solutions for the Navier-Stokes equations, we refer the reader to \cite{Foias72, VF78, VF88, FMRT} and also to the more recent paper \cite{FRT2013}. Except for the fact that we allow for a slightly less regular external forcing term, the existence result that we obtain in this section has been already proved in \cite{FRT2013}, a work that has been in fact our inspiration. Nevertheless, we want to show that our abstract framework does apply to this case, with a much simpler proof than that presented in \cite{FRT2013}.

For the sake of simplicity, we assume periodic boundary conditions. In this case we consider a periodic domain given by $\Omega=\Pi_{i=1}^3 (0,L_i)$, where $L_i>0$, for $i=1,2,3$. This means we are assuming the flow is periodic with period $L_i$ in each spatial direction $x_i$. We also consider the averages of the flow and of the forcing term to be zero, i.e.,
\[\int_\Omega\bu(\bx,t)\;\rd\bx=0\,,\,\,\,\,\,\int_\Omega \f(\bx,t)\;\rd \bx=0.\]

Let $C^{\infty}_{per}(\Omega,\mathds R^3)$ denote the space of infinitely
differentiable and $\Omega$-periodic functions $\bu$. We define the set of periodic test functions with vanishing average and divergence free as
\be\label{eqdefmV}
V^\infty:=\left\{\bu\in \mC^{\infty}_{per}(\Omega,\mathds R^3) \,\left|\, \nabla \cdot \bu=0\mbox{ and }\int_\Omega \bu(\bx)\;\rd\bx=0\right.\right\}.
\ee

Let $H$ be the closure of $V^\infty$ in $L^2(\Omega,\mathds R^3)$ and let $V$ be the closure of $V^\infty$ in $H^1(\Omega,\mathds R^3)$.
The inner product and norm in $H$ are defined, respectively, by
\[
(\bu,\bv)=\int_{\Omega}\bu\cdot \bv \;\rd\bx \;\mbox{ and } \quad |\bu| = \sqrt{(\bu,\bu)},
\]
where $\bu\cdot\bv = \sum_{i=1}^3 u_iv_i$. In the space $V$, these are defined as
\[
(\!(\bu, \bv)\!) = (\nabla \bu,\nabla \bv) = \int_\Omega \nabla \bu : \nabla \bv \;\rd \bx \;\mbox{ and } \quad \|\bu\| = \sqrt{(\!(\bu,\bu)\!)},
 \]
where it is understood that $\nabla \bu=(\partial u_i/\partial x_j)_{i,j=1}^3$ and that $\nabla \bu : \nabla \bv$ is the componentwise product between $\nabla \bu$ and $\nabla \bv$. We also consider the space $H$ endowed with its weak topology and denote it by $H_\rw$.

Clearly, $V$ is a subset of $H$. Thus, by identifying $H$ with its dual space $H'$, we obtain the following continuous inclusions
\[ V \subset H \equiv H' \subset V'.
\]
Since the injection $H \hookrightarrow V'$ is a continuous linear mapping, we also have
\[
H_\rw \hookrightarrow V'_{w*}
\]
with continuous injection, where $V'_{w*}$ denotes the space $V'$ endowed with the weak-star topology.

Let $A$ be the \textbf{Stokes operator}, defined as $A=-\mathbb P \Delta$, where $\mathbb P: L^2(\Omega,\mathbb R^3) \rightarrow H$ is the Leray-Helmholtz projection, i.e., the orthogonal projector in  $L^2(\Omega,\mathbb R^3)$ onto the subspace of divergence-free vector fields. We denote by $D(A)$  the domain of $A$, which is defined as the set of functions $\bu \in V$ such that $A\bu\in H$. In the periodic case with zero
average, we have
\[A\bu=-\Delta \bu\,,\,\,\forall \bu \in D(A)= V\cap H^2(\Omega,\mathbb R^3)\]
and $A$ is a positive self-adjoint linear operator with compact inverse, so that it has a nondecreasing sequence $\{\lambda_i\}_{i\in \mathds N}$ of positive eigenvalues counted according to their multiplicity, associated with an orthonormal basis $\{\bw_i\}_{i\in \mathds N}$ in $H$. Furthermore, the Poincar\'e inequality holds, i.e., for all $\bu\in V$,
\begin{eqnarray}\label{poincare}
\lambda_1|\bu|^2\leq \|\bu\|^2,
\end{eqnarray}
where $\lambda_1 > 0$ is the first eigenvalue of the Stokes operator.

The Stokes operator $A$ extends to an operator defined on $V$, with values in $V'$, and so that
\[
\|A\bu\|_{V'} = \sup_{\|\bv\|=1} \langle A\bu,\bv\rangle_{V',V} = \sup_{\|\bv\|=1} (\!(\bu,\bv)\!) = \|\bu\|,
\]
which implies in particular that $A: V\rightarrow V'$ is continuous.

We denote by $P_k: H \rightarrow V$ the Galerkin projector onto the space spanned by the eigenfunctions associated with the first $k$ eigenvalues of the Stokes operator, i.e.
\[
P_k \bu = \sum_{i=1}^k (\bu,\bw_i) \bw_i, \quad \forall \bu \in H.
\]
Since $\{\bw_i\}_{i\in \mathds N}$ is an orthonormal basis in $H$, we obtain
\[
\|P_k \bu\|^2 = (\!( P_k \bu, P_k \bu )\!) = \sum_{j=1}^k |(\bu,\bw_j)|^2 (\!(\bw_j, \bw_j)\!) = \sum_{j=1}^k |(\bu,\bw_j)|^2 \lambda_j \leq \lambda_k |\bu|^2,
\]
for every $\bu\in H$, from which it follows that $P_k$ is continuous from $H$ to $V$.

Let $B: V\times V \rightarrow V'$ be the bilinear operator
\[
B(\bu,\bv) = \mathbb P[(\bu\cdot \nabla)\bv], \quad \forall \bu,\bv \in V,
\]
which satisfies the inequality (see \cite{bookcf1988,FMRT})
\be\label{estB}
\| B(\bu,\bv)\|_{V'} \leq c|\bu|^{1/4}\|\bu\|^{3/4}|\bv|^{1/4}\|\bv\|^{3/4}, \quad \forall \bu,\bv\in V,
\ee
where $c$ is a universal constant. By using this inequality, it is not difficult to see that $B:V\times V \rightarrow V'$ is also a continuous operator.

The natural space for the solutions of the Navier-Stokes equations is the space $\Cloc(I,H_\rw)$ of continuous functions
from an interval $I\subset \mathds R$ to $H_\rw$, endowed with the compact-open topology. This function space is the space of weakly continuous functions from  $I$ to $H$.

The Navier-Stokes equations can be written in the functional form
\be\label{NS}
\bu_t+\nu A\bu+B(\bu,\bu)=\f \quad \text{in } V',
\ee
in the sense of Bochner (see Remark \ref{bochnerint}), which in this case is equivalent to the weak formulation
\[ \frac{\rd}{\rd t} (\bu,\bv) + \nu (\!(\bu,\bv)\!) + ((\bu\cdot\nabla)\bu,\bv) = (\f,\bv),
\]
in the distribution sense on $I$, for any $\bv\in V$.


The notion of solution that is considered here is the well-known Leray-Hopf weak solution, which is defined below.

\renewcommand{\theenumi}{\roman{enumi}}
\begin{defs}\label{NSweak}Let $I$ be an interval in $\mathds R$ and $\f\in
L_{\rloc}^2(I, V')$. We say that $\bu$ is a \textbf{Leray-Hopf weak solution of the Navier-Stokes equations \eqref{nse1}-\eqref{nse2} on $I$} if
\begin{enumerate}
 \item $\bu\in L^{\infty}_{\rloc}(I, H)\cap L_{\rloc}^2(I, V)\cap \mathcal
C_{\rloc}(I,H_\rw)$;

 \item $\partial_t\bu \in   L_{\rloc}^{4/3}(I, V')$;

 \item \label{NSweakeq} $\bu$ satisfies the functional formulation of the Navier-Stokes
equations
\be\label{eq3b}\bu_t+\nu A\bu+B(\bu,\bu)=\f;\ee

 \item $\bu$ satisfies the energy inequality in the sense that for almost
all $t'\in I$ and for all $t\in I$ with $t>t'$,
\begin{equation}\label{energy-ineq}
\frac{1}{2}
|\bu(t)|^2+\nu\int_{t'}^{t}\|\bu(s)\|^2\;\rd s \leq \frac{1}{2}
|\bu(t')|^2+
\int_{t'}^{t}\langle\f(s),\bu(s)\rangle_{V',V}\;\rd s;
\end{equation}

 \item \label{nsedefweaksolv} If $I$ is closed and bounded on the left, with left end point
$t_0$, then the solution is strongly continuous in $H$ at $t_0$ from the right,
i.e., $\bu(t)\rightarrow \bu(t_0)$ in $H$ as $t \rightarrow t_0^+$.
\end{enumerate}
\end{defs}

The set of allowed times $t'$ in \eqref{energy-ineq} can be characterized as the points of strong continuity from the right of $\bu$ in $H$. In particular, condition \eqref{nsedefweaksolv} implies that $t'=t_0$ is allowed in that case.

The Leray-Hopf weak solutions of the Navier-Stokes equations also satisfy a strenghtened form of the energy inequality \eqref{energy-ineq} of Definition \ref{NSweak}. The proof of this strenghtened energy inequality has been given in \cite{FMRT} for external forces $\f$ in $L^2_{\rloc}(I,H)$. This inequality is also valid if $\f$ belongs to the larger space $L^2_{\rloc}(I,V')$, as we state below:

\begin{prop}\label{prop-strenghtened-energy-ineq}
Let $T>0$ and $\f \in L^2(0,T;V')$. Consider a nonnegative, nondecreasing and continuously-differentiable real-valued function  $\psi:[0,\infty)\rightarrow\mathbb{R}$ with bounded derivative. If $\bu$ is a Leray-Hopf weak solution of the Navier-Stokes equations on $[0,T]$, then
\[
\frac{\rd}{\rd t}(\psi(|\bu(t)|^2))\leq 2\psi'(|\bu(t)|^2)[\langle \f(t),\bu(t)\rangle_{V',V}-\nu\|\bu(t)\|^2]
\]
in the sense of distributions on $[0,T]$.
\end{prop}

The idea of the proof is to first obtain such inequality for mollifications of $\psi$ and $|\bu(t)|^2$, and then to pass to the limit with respect to the parameters of each mollification, in a suitable order. For the detailed proof, see \cite{TeseCecilia}.

Given $R > 0$, we denote by $B_H(R)$ the closed ball centered at the origin and with radius $R$ in $H$. The corresponding closed ball endowed with the weak topology is denoted by $B_H(R)_\rw$. We define the following sets of Leray-Hopf weak solutions:
\be\label{u}
\mU_I=\{\bu\in \Cloc(I,H_\rw): \bu \mbox{ is a Leray-Hopf
weak solution on } I\},
\ee
\be
\mU_I(R)=\{\bu\in  \Cloc(I,B_H(R)_\rw): \bu \mbox{ is a Leray-Hopf
weak solution on } I\},
\ee
\be
\mU^\sharp_I=\{\bu\in \Cloc(I,H_\rw): \bu \mbox{ is a Leray-Hopf
weak solution on } \mathring{I}\},
\ee
\be\label{defmUIsharpR}
\mU^\sharp_I(R)=\{\bu\in  \Cloc(I,B_H(R)_\rw): \bu \mbox{ is a
Leray-Hopf weak solution on } \mathring{I}\},
\ee
where $\mathring{I}$ denotes the interior of the interval $I$.

From \eqref{energy-ineq}, one obtains the classical estimates
\be\label{ener-est}|\bu(t)|^2\leq|\bu(t')|^2+\frac{1}{\nu}\|\f\|^2_{L^2(t',t;V')},\ee
\be\label{nsest2}\int_{t'}^t\|\bu(s)\|^2\;\rd s\leq\frac{1}{\nu}|\bu(t')|^2+\frac{1}{\nu^2}\|\f\|^2_{L^2(t',t;V')},\ee
valid for any $\bu\in \mU_I$, for any time $t'\in I$ allowed in \eqref{energy-ineq} and for any $t\in I$ with $t\geq t'$. Moreover, using \eqref{estB}, one also has that
\be\label{nsest3}
\left(\int_{t'}^t\|\partial_t\bu(s)\|_{V'}^{4/3}\;\rd s\right)^{3/4}\leq \frac{c}{\nu^{3/4}}|\bu(t')|^2+\frac{\nu^{5/4}}{\lambda_1^{1/2}}D(t',t),\ee
where $c$ is a universal constant and $D(t',t)$ is a nondimensional function which depends on the variables $t',t$ and also on the parameters $\nu$, $\lambda_1$, $\Omega$, and $\f$ through the nondimensional quantities $\nu\lambda_1|t-t'|$ and $(\lambda_1^{1/4}/\nu^{3/2})\|\f\|_{L^2(t',t;V')}$.

The a~priori estimates \eqref{ener-est}-\eqref{nsest3} allow us to prove that $\mU_I^\sharp(R)$ is a compact and metrizable space, in the same way as it was done in \cite[Proposition 2.2]{FRT2013}. Furthermore, one can show that $\mU_I^\sharp(R)$ is the closure of the space $\mU_I(R)$ with respect to the topology of $\mC(I,H_\rw)$.

The existence of a Leray-Hopf weak solution on a given interval $I\subset \mathds{R}$ is obtained using the estimates \eqref{ener-est}, \eqref{nsest2}, and \eqref{nsest3}. This proof is a classical result and can be found in many well-known texts \cite{bookcf1988,Lady,Lions,temam1995}. We state it below for completeness.

\begin{thm}\label{nsexistence}
Let $I\subset\mathds{R}$ be an interval closed and bounded on the left with left end point $t_0$ and let $f\in L^2_{\rloc}(I,V')$. Then, given $\bu_0 \in H$, there exists at least one weak solution $\bu\in \mU_I$ of \eqref{nse1}-\eqref{nse2} in the sense of Definition \ref{NSweak} satisfying $\Pi_{t_0}\bu = \bu_0$.
\end{thm}


From now on, we assume that $I\subset\mathds{R}$ is an interval closed and bounded on the left, with left end point $t_0$. Under this assumption, the energy inequality \eqref{energy-ineq} is valid for $t'=t_0$.

Consider a compact subinterval $J \subset I$. Then given $\bu \in \mU_I$ such that $\bu(t_0) \in B_H(R)$ for some $R\geq 0$, it follows, from \eqref{ener-est} with $t'=t_0$, that there exists $\tilde{R} \geq R$ such that $\bu(t) \in B_H(\tilde{R})$, for every $t \in J$. Thus, the restriction of $\bu$ to $J$ belongs to $\mU_J(\tilde{R})$.

In order to prove the existence of a trajectory statistical solution for the Navier-Stokes equations satisfying a given initial data, we shall apply Theorem \ref{existencestatsol} by considering $X$ as the space $H_\rw$ and the general set $\mU$ as the set of weak solutions $\mU_I$. We now show that $\mU_I$ satisfies the hypotheses of Theorem \ref{existencestatsol}.

First, note that hypothesis \eqref{eqH1b} is a direct consequence of Theorem \ref{nsexistence}. Also, defining $\fK'(H_\rw)$ as the family of (strongly) compact sets in $H$, hypothesis \eqref{eqH2bii} follows from the fact that $H$ is a separable Banach space (see Sections \ref{subsecelementsofmeastheory} and \ref{subsecinjectborel}). The only hypothesis that needs more work is the remaining hypothesis \eqref{eqH2biii}, which is the subject of the next proposition.

\begin{prop}\label{prop-KUI-compact}
Let $I \subset \mathds{R}$ be an interval closed and bounded on the left with left end point $t_0$ and let $K$ be a strongly compact set in $H$. Then $\KUI$ is a compact set in $\mX = \Cloc(I,H_\rw)$.
\end{prop}
\begin{proof}
Let $\bu\in\KUI$ and let $R\geq 0$ be sufficiently large so that $K\subset B_H(R)$. Consider an increasing sequence $\{J_n\}_n$ of compact subintervals of $I$ with left end point at $t_0$ and such that $I = \bigcup_n J_n$. Since $\Pi_{t_0}\bu\in K\subset B_H(R)$ and $\bu\in\mU_I$, it follows, from the estimate \eqref{ener-est} with $t'=t_0$, that there exists a sequence $\{R_n\}_n$ of positive real numbers such that $\Pi_{J_n}\bu\in\mU_{J_n}(R_n)$, for every $n$. Thus,
\be\label{nseq1}
\Pi_{t_0}^{-1}K\cap\mU_I\subset\bigcap_n\Pi_{J_n}^{-1}\mU_{J_n}(R_n).
\ee
Since each $\mU_{J_n}(R_n)$ is a metrizable space, \eqref{nseq1} implies that $\KUI$ is also metrizable. Therefore, it suffices to show that $\KUI$ is sequentially compact.

Let $\{\bu_k\}_k$ be a sequence in $\KUI$. As in the classical proof of existence of weak solutions (Theorem \ref{nsexistence}), using the a~priori estimates \eqref{ener-est}-\eqref{nsest3} on each compact interval $J_n$ and applying a diagonalization method, we obtain a subsequence $\{\bu_{k'}\}_{k'}$ and a function $\bu$ such that
\be \label{nseconvsubseq}
\bu_{k'} \rightarrow \bu \quad \mbox{in } \Cloc(I,H_\rw),
\ee
as $k'\rightarrow \infty$. Moreover, this limit function $\bu$ is a weak solution on the interior of $I$, i.e. $\bu \in \mU_I^{\sharp}$ (the condition of strong continuity at $t_0$, item \eqref{nsedefweaksolv} of Definition \ref{NSweak}, is not guaranteed at this point). From \eqref{nseconvsubseq} we obtain in particular that
\be\label{nseconvinittime}
\bu_{k'}(t_0) \rightarrow \bu(t_0) \quad \text{in } H_\rw.
\ee
On the other hand, since $K$ is a compact set in $H$, there exists a further subsequence, which we still denote by $\{\bu_{k'}\}_{k'}$, and an element $\bu_0 \in K$ such that
\be\label{nseq3}
\bu_{k'}(t_0) \rightarrow \bu_0 \quad \text{in }H.
\ee
From \eqref{nseconvinittime} and \eqref{nseq3} it follows that $\bu(t_0)=\bu_0$, which implies that $\bu\in \Pi_{t_0}^{-1}K$. Moreover, we obtain that $\{\bu_{k'}(t_0)\}_{k'}$ also converges to $\bu(t_0)$ in the strong topology of $H$. 

This allow us to prove that $\bu$ verifies in addition the last condition of Definition \ref{NSweak}. Indeed, since each $\bu_{k'}$ belongs to $\mU_I$, they satisfy in particular the energy inequality \eqref{energy-ineq} at $t'=t_0$. Considering the $\liminf$ as $k' \rightarrow \infty$ in this inequality we obtain, by using the strong convergence of $\{\bu_{k'}(t_0)\}_{k'}$ to $\bu(t_0)$ in $H$, that
\[
\frac{1}{2}
|\bu(t)|^2+\nu\int_{t_0}^{t}\|\bu(s)\|^2\;\rd s \leq \frac{1}{2}
|\bu(t_0)|^2+\int_{t_0}^{t}\langle\f(s),\bu(s)\rangle_{V',V}\;\rd s.
\]
Then, by taking the $\limsup$ as $t \rightarrow t_0^+$ above, we obtain
\be\label{nselimsupinittime}
\limsup_{t \rightarrow t_0^+} |\bu(t)|^2 \leq |\bu(t_0)|^2.
\ee
Since $\bu \in \Cloc(I,H_\rw)$, we also have that
\be\label{nseliminfinittime}
|\bu(t_0)|^2 \leq \liminf_{t \rightarrow t_0^+} |\bu(t)|^2.
\ee
Now \eqref{nselimsupinittime} and \eqref{nseliminfinittime} imply that $\bu(t)$ converges in norm to $\bu(t_0)$ as $t \rightarrow t_0^+$. Since $\bu(t)$ also converges weakly to $\bu(t_0)$ as $t \rightarrow t_0^+$, we deduce that
\[
\lim _{t \rightarrow t_0^+} \bu(t) = \bu(t_0) \quad \mbox{in } H.
\]
This means that $\bu\in\KUI$, which completes the proof of compactness of this set.
\end{proof}

Now we are able to prove the existence of a solution for the corresponding Initial Value Problem \ref{ivp_tss} associated with the Navier-Stokes equations.

\begin{thm}\label{thm-existence-trajss-NSE}
Let $I\subset\mathds{R}$ be an interval closed and bounded on the left with left end point $t_0$ and let $\mU_I$ be the set of Leray-Hopf weak solutions of the Navier-Stokes equations on $I$. Then, given a Borel probability measure $\mu_0$ on $H$, there exists a $\mU_I$-trajectory statistical solution $\rho$ on $\Cloc(I,H_\rw)$ satisfying the initial condition $\Pi_{t_0}\rho = \mu_0.$
\end{thm}

\begin{proof}
Our intention is to apply Theorem \ref{existencestatsol} to the set $\mU_I$, with $X=H_\rw$ and $\fK'(H_\rw)$ as the family of (strongly) compact sets in $H$. First of all, since $H$ is a separable Banach space, the Borel sets in $H$ and $H_\rw$ coincide (see Section \ref{subsecinjectborel}). This implies that any Borel probability measure $\mu_0$ on $H$ is also a Borel probability measure on $H_\rw$, and vice-versa, so that we can refer indistinguishably to probability measures on either $H$ or $H_\rw$. Moreover, since $H$ is a Polish space, it follows that any Borel probability measure on $H$ is tight in the sense of being inner regular with respect to the family of compact subsets of $H$ (\cite[Theorem 12.7]{AB}). Thus, $\fK'(H_\rw)$ satisfies the hypothesis \eqref{eqH2bii} of Theorem \ref{existencestatsol}. From Theorem \ref{nsexistence}, it follows that the set $\mU_I$ satisfies hypothesis \eqref{eqH1b} of Theorem \ref{existencestatsol}. From Proposition \ref{prop-KUI-compact}, we also obtain that $\mU_I$ satisfies hypothesis \eqref{eqH2biii} of Theorem \ref{existencestatsol}. Thus, $\mU_I$ verifies all the hypotheses of Theorem \ref{existencestatsol}. Therefore, we apply Theorem \ref{existencestatsol} to deduce that there exists a $\mU_I$-trajectory statistical solution $\rho$ with $\Pi_{t_0}\rho=\mu_0$.
\end{proof}

Finally, using Theorem \ref{thm-existence-trajss-NSE} and the strengthened energy inequality from Proposition \ref{prop-strenghtened-energy-ineq}, we obtain a solution for the corresponding Initial Value Problem \ref{ivp_ss} associated with the Navier-Stokes equations. More specifically, we prove the existence of a projected statistical solution of the Navier-Stokes equations, in the sense of Definition \ref{defprojectstatsol}, associated with a $\mU_I$-trajectory statistical solution and satisfying a given initial data.

\begin{thm}\label{thm-existence-ss-NSE}
Let $I\subset\mathds{R}$ be an interval closed and bounded on the left with left end point $t_0$ and let $\mU_I$ be the set of Leray-Hopf weak solutions of the Navier-Stokes equations on $I$. Consider a Borel probability measure $\mu_0$ on $H$ satisfying
\be\label{eq-finite-energy}
\int_H |\bu|^2\;\rd \mu_0(\bu) < \infty.
\ee
Then there exists a projected statistical solution $\{\rho_t\}_{t\in I}$ of the Navier-Stokes equations \eqref{nse1}-\eqref{nse2}, associated with a $\mU_I$-trajectory statistical solution, such that
\renewcommand{\theenumi}{\roman{enumi}}
\begin{enumerate}
  \item \label{meaninitialconditionnse} The initial condition $\rho_{t_0}=\mu_0$ holds;
  \item \label{continuitystatinfonse} The function
\be\label{intphinse}
t\mapsto \int_H \varphi(\bu)\;\rd\rho_t(\bu)
\ee
is continuous on $I$, for every bounded and weakly-continuous real-valued function $\varphi$ on $H$, and is measurable on $I$, for every bounded and continuous real-valued function $\varphi$ on $H$.
  \item \label{meaneqnsecondition} For any cylindrical test function $\Phi$ in $V'$, it follows that
\begin{multline}
  \label{meaneqnse}
\int_H\Phi(\bu)\;\rd\rho_t(\bu)=\int_H\Phi(\bu)\;\rd\rho_{t'}(\bu)\\
+\int_{t'} ^t\int_H\langle \f(s) - \nu A\bu - B(\bu,\bu),\Phi'(\bu)\rangle_{V',V}\;\rd\rho_s(\bu)\rd s,
\end{multline}
for all $t,t'\in I$.
  \item\label{existssiv} The mean strengthened energy inequality
  \be\label{meanstrengthenedenergyineq}
  \frac{\rd}{\rd t}\int_H (\psi(|\bu|^2)) \;\rd\rho_t(\bu) \leq 2\int_H \psi'(|\bu|^2)[\langle \f(t),\bu\rangle_{V',V}-\nu\|\bu\|^2]\;\rd\rho_t(\bu)
  \ee
  is satisfied in the distribution sense on $I$, for every nonnegative, nondecreasing and continuously-differentiable real-valued function $\psi$ with bounded derivative.
  \item \label{meancontorigin} At the initial time, the limit
  \be\label{eqmeancontorigin} \lim_{t\rightarrow t_0^+} \int_H \psi(|\bu|^2) \;\rd\rho_t(\bu) = \int_H \psi(|\bu|^2) \;\rd\mu_0(\bu)
  \ee
  holds for every function $\psi$ as in \eqref{existssiv}.
\end{enumerate}
\end{thm}
\begin{proof}
We have seen in the proof of Theorem  \ref{thm-existence-trajss-NSE} that the set of Leray-Hopf weak solutions $\mU_I$ satisfies all the hypotheses \eqref{eqH1b}, \eqref{eqH2bii}, and \eqref{eqH2biii} of Theorem \ref{existencestatsol}. Now let $\F:I\times V\rightarrow V'$ be the function defined by
\[\F(t,\bu)= \f(t)-\nu A\bu-B(\bu,\bu).\]
As previously mentioned, the linear operator $A:V\rightarrow V'$ and the bilinear operator $B: V\times V \rightarrow V'$ are continuous. This implies that the mapping $\bu \mapsto -\nu A\bu - B(\bu,\bu)$ is also continuous from $V$ into $V'$. In particular, the mapping $(t,\bu) \mapsto -\nu A\bu - B(\bu,\bu)$ is $(\fL_I\otimes \fB_V,\fB_{V'})$-measurable. Further, since $\f \in L^2_{\rloc}(I,V')$, we obtain that $\F$ is a $(\fL_I\otimes \fB_V,\fB_{V'})$-measurable function. From the functional equation \eqref{NS} and the fact that any weak solution $\bu\in\mU_I$ satisfies $\bu_t\in L^{4/3}(I;V')$, it follows that \eqref{Ftuislocint} holds. From the condition \eqref{NSweakeq} of the Definition \ref{NSweak} of a Leray-Hopf weak solution, the validity of $\bu_t = \F(t,\bu(t))$ in the sense of Bochner implies that \eqref{evoleqweak} holds in the weak sense (see Remark \ref{bochnerint}). Hence, hypothesis \eqref{HuteqinmU} of Theorem \ref{thm-existence-projectedss} holds.

The a~priori estimate \eqref{nsest3} means that there exists a function $\gamma:I\times H_\rw\rightarrow\mathds{R}$ such that
\[\int_{t_0}^t\|\F(s,\bu(s))\|_{V'}\;\rd s \leq \gamma(t,\bu(t_0)),\quad \forall t\in I, \quad \forall \bu\in \mU_I,\]
which is clearly $(\fL_I \otimes \fB_{H_\rw})$-measurable and, thanks to \eqref{eq-finite-energy}, with
\[t\mapsto \int_{H_\rw} \gamma(t,\bu_0)\;\rd \mu_0(\bu_0) \quad \text{in } L^1_{\rloc}(I). 
\]
Thus, hypothesis \eqref{HFgamma} and condition \eqref{gammaL1} hold (see Remark \ref{rmkconditionwithnormyprime}).

Also, as mentioned before, we know that $V \subset H_\rw \subset V'_{w*}$, with all the injections being continuous. Since $H$ is a separable Banach space, then $\fB_{H_\rw} = \fB_H$. Moreover, since $V$ is a Polish space, then $\fB_V \subset \fB_H = \fB_{H_\rw}$ (see Section \ref{subsecinjectborel}), showing that hypothesis \eqref{Hfbzinfbx} is also verified. 

Then, applying Theorem \ref{thm-existence-projectedss} with $X=H_\rw$, $Z=Y=V$, $\mU=\mU_I$, $\F$ and $\gamma$ as above, we obtain the existence of a projected statistical solution $\{\rho_t\}_{t\in I}$ associated with a $\mU_I$-trajectory statistical solution $\rho$ and such that $\rho_{t_0}=\mu_0$. This means that $\{\rho_t\}_{t\in I}$ satisfies \eqref{meaninitialconditionnse}, \eqref{meaneqnsecondition}, and the first part of \eqref{continuitystatinfonse}, concerning bounded and weakly-continuous functions $\varphi$ on $H$.

Let us prove the second part of property \eqref{continuitystatinfonse}, concerning strongly continuous functions. Consider a bounded and continuous real-valued function $\varphi$ on $H$. Let $P_m$, $m\in\mathds{N}$, be the Galerkin projectors. Then, for every $m\in\mathds{N}$, the function $\varphi\circ P_m$ is bounded and continuous on $H_\rw$. Let $\mV\subset\mU$ be a Borel subset such that $\rho(\mV)=1$. From the first part of \eqref{continuitystatinfonse}, it follows that the function
\[t\mapsto \int_{\mV} \varphi(P_m\bu(t))\;\rd\rho(\bu)
\]
is continuous on $I$, for every $m\in\mathds{N}$. Then, since the function \eqref{intphinse} is the pointwise (in $t$) limit of these functions as $m\rightarrow \infty$, it follows that \eqref{intphinse} is measurable on $I$. This proves the second part of \eqref{continuitystatinfonse}.

For the proof of \eqref{existssiv}, consider the functions $\alpha:I\times H_\rw\rightarrow\mathds{R}$ and $\beta:I\times V\rightarrow \mathds{R}$ defined respectively by
\be\label{eq-alpha}\alpha(t,\bu(t))=\psi(|\bu(t)|^2),\ee
and
\be\label{eq-beta}\beta(t,\bu(t)) = - 2\psi'(|\bu(t)|^2)[\langle f(t),\bu(t) \rangle_{V',V}-\nu\|\bu(t)\|^2],\ee
for every $\bu\in\mU_I$ and $t\in I$. Using the Galerkin projector as above, we see that $\alpha$ and $\beta$ are the pointwise limit of continuous functions, hence they are measurable maps as required in Proposition \ref{prop-energy-ineq}. Using the estimates \eqref{ener-est} and \eqref{nsest2} with $t'=t_0$, which is allowed for functions in $\mU_I$, and using \eqref{eq-finite-energy}, we obtain that $\alpha\in L^\infty(J\times\mV,\lambda\times\rho)$ and $\beta \in L^1(J\times\mV,\lambda\times\rho)$, for every compact subset $J\subset I$, where $\lambda$ denotes the Lebesgue measure on $I$. From Proposition \ref{prop-strenghtened-energy-ineq}, the functions $\alpha$ and $\beta$ satisfy
\[\frac{\rd}{\rd t}\alpha(t,\bu(t)) + \beta(t,\bu(t)) \leq 0,\quad \forall \bu\in\mU_I,\]
in the sense of distributions in $I$. Property \eqref{existssiv} then follows by applying Proposition \ref{prop-energy-ineq} with $X=H_\rw$, $Y=V$, $\mU=\mU_I$ and with the functions $\alpha$ and $\beta$ defined in \eqref{eq-alpha}-\eqref{eq-beta}.

It only remains to prove property \eqref{meancontorigin}. Note that for every function $\psi$ as in \eqref{existssiv}, we may write
\[\psi(|\bu(t)|^2)\leq \psi(0)+\psi'(\xi)|\bu(t)|^2,\]
for some $0 \leq \xi \leq |\bu(t)|^2$. Then, using the boundedness of $\psi'$ and the a~priori estimate \eqref{ener-est} with $t'=t_0$, we find that
\[  \psi(|\bu(t)|^2)\leq C_0 + C_1|\bu(t_0)|^2,
\]
for suitable constants $C_0, C_1>0$. Hence, from \eqref{eq-finite-energy}, it follows that $\psi(|\bu(t)|^2)$ is bounded by a $\rho$-integrable function in $\mV$ which does not depend on $t$. Furthermore, since every $\bu\in \mU_I$ is strongly continuous at $t_0$ and $\psi$ is continuous, we have that
\[ \psi(|\bu(t)|^2) \rightarrow \psi(|\bu(t_0)|^2),
\]
$\rho$-almost everywhere, as $t\rightarrow t_0^+$. Therefore, \eqref{eqmeancontorigin} follows from the Lebesgue Dominated Convergence Theorem.
\end{proof}

\subsection{Reaction-Diffusion Equation}
\label{subsecreacdiffeq}

In this section, we consider the following reaction-diffusion-type equation

\be\label{reacdiffeq1}
\frac{\partial u}{\partial t}(\bx,t) = a\Delta u(\bx,t) - f(t,u(\bx,t)) + g(\bx,t), \quad \bx \in \Omega \subset \mathds{R}^n, \quad t\in I,
\ee
subject to the boundary condition
\be\label{reacdiffeq2}
u(\bx,t)|_{\bx\in\partial\Omega}=0, \quad \forall t\in I,
\ee
where $u$ is the unknown variable, $a$ is a positive constant, $f$ is the reaction function and $g$ is the external force. Moreover, $I \subset \mathds R$ is an arbitrary interval and $\Omega \subset \mathds R^n $ is a bounded and open subset which is assumed to be smooth.

We follow the same framework and notations from \cite[Section XV.3]{ChepVishik}, but in order to simplify the presentation we consider only a scalar equation instead of a system of equations.

Consider the spaces $H=L^2(\Omega)$ and $V=H_0^1(\Omega)$ with respective norms $|\cdot|_H$ and $\|\cdot\|_V$, given by
\[|v|_H^2=\int_\Omega|v(\bx)|^2\;\rd \bx\,,\,\,\forall v\in H,\]
and
\[\|v\|_V^2=\int_\Omega|\nabla v(\bx)|^2\;\rd \bx\,,\,\,\forall v\in V.\]
Also, consider $V'=H^{-1}(\Omega)$, the dual of $H_0^1(\Omega)$, with duality product $\langle \cdot, \cdot \rangle_{V',V}$.  Then, identifying $H$ with its dual space $H'$, we have
\[
V \subset H \equiv H'\subset V',
\]
with continuous inclusions and, in particular, $H_\rw \hookrightarrow V'_{w*}$ with continuous injection. We also consider the space $H_0^r(\Omega)$, for $r>0$, and its dual $H^{-r}(\Omega)$, with $H_0^r(\Omega) \subset H \subset H^{-r}(\Omega)$, for every $r>0$. For $r \geq n/2 - n/p$ and $p\geq 2$, we have $H_0^r(\Omega)\subset L^p(\Omega)$. We denote the duality product between $L^p(\Omega)$ and $L^q(\Omega)$, $1\leq p,q \leq \infty$, $1/p+1/q=1$, simply by $(\cdot,\cdot)$, which includes the inner product of $H$.

We assume that $g\in L^2_{\rloc}(I,V')$ and that $f$ is a function in $\mC(\mathds{R}\times\mathds{R},\mathds{R})$ satisfying the following estimates, for every $v\in\mathds{R}$ and $s\in\mathds{R}$:
\be\label{reacdiffcondf1}
\eta|v|^p-C_1 \leq f(s,v)v,
\ee
\be\label{reacdiffcondf2}
|f(s,v)|^{\frac{p}{p-1}} \leq C_2(|v|^p+1),
\ee
where $\eta>0$, $p\geq 2$ and $C_1$, $C_2\in\mathds{R}$ are constants. In \cite{ChepVishik}, the function $g$ is also assumed to be translation bounded in the space $L^2_{\rloc}(I,V')$, but we do not make this assumption  since we do not need uniform estimates for arbitrarily large times.

As in \cite{ChepVishik}, it follows by using condition \eqref{reacdiffcondf2} that if $r\geq\max\{1,n(1/2-1/p)\}$ and $u\in L^p_{\rloc}(I,L^p(\Omega))\cap L^2_{\rloc}(I,V)$, then $\partial_t u\in L^q_{\rloc}(I,H^{-r}(\Omega))$, for $1/p + 1/q = 1$. This implies that the evolution equation \eqref{reacdiffeq1} can be considered in the distribution sense on $I$, with values in $H^{-r}(\Omega)$. 

We then have the following definition of a weak solution for problem \eqref{reacdiffeq1}-\eqref{reacdiffeq2}.

\begin{defs}\label{def-wsol-reacdiff}
A weak solution of \eqref{reacdiffeq1}-\eqref{reacdiffeq2} is a function $u=u(\bx,t)$ on $\Omega\times I$ such that $u\in L^p_{\rloc}(I,L^p(\Omega))\cap L^2_{\rloc}(I,V)$ and $u$ satisfies \eqref{reacdiffeq1} in the distribution sense on $I$, with values in $H^{-r}(\Omega)$.
\end{defs}

Given $R \geq 0$, let $B_H(R)$ be the closed ball centered at the origin and of radius $R$ in $H$. Consider the following sets of weak solutions:
\begin{align}
\mU_I & =\{u\in\Cloc(I,H)\,|\,u\text{ is a weak solution of \eqref{reacdiffeq1}-\eqref{reacdiffeq2} on }I\}, \\
\mU_I(R) & =\{u\in\Cloc(I,B_H(R))\,|\,u\text{ is a weak solution of \eqref{reacdiffeq1}-\eqref{reacdiffeq2} on }I\}.
\end{align}

The proof of existence of individual weak solutions for the corresponding initial value problem of \eqref{reacdiffeq1}-\eqref{reacdiffeq2} can be found in \cite{ChepVishik}. We state it below for completeness.

\begin{thm}\label{existence-reacdiff}
Consider an interval $I\subset\mathds{R}$ bounded and closed on the left with left end point $t_0$. Let $g\in L^2_{\rloc}(I,V')$ and let $f\in\mC(\mathds{R}\times\mathds{R},\mathds{R})$ be a function satisfying conditions \eqref{reacdiffcondf1} and \eqref{reacdiffcondf2}. Then, given $u_0\in H$, there exists a weak solution $u$ of problem \eqref{reacdiffeq1}-\eqref{reacdiffeq2} such that $u\in L^p_{\rloc}(I,L^p(\Omega))\cap L^2_{\rloc}(I,V)\cap L^{\infty}_{\rloc}(I,H)$ and $u(t_0)=u_0$.
\end{thm}

The following proposition presents some additional properties satisfied by every weak solution of \eqref{reacdiffeq1}-\eqref{reacdiffeq2} in the sense of Definition \ref{def-wsol-reacdiff}. The proof is given in \cite[Proposition XV.3.1]{ChepVishik}.

\begin{prop}
Let $u \in L^p_{\rloc}(I,L^p(\Omega))\cap L^2_{\rloc}(I,V)$ be a weak solution of \eqref{reacdiffeq1}-\eqref{reacdiffeq2}. Then
\renewcommand{\theenumi}{\roman{enumi}}
\begin{enumerate}
\item $u \in \Cloc(I,H)$;
\item the function $|u(s)|_H^2$ is absolutely continuous on every compact subinterval $J \subset I$ and satisfies the following energy equality
\be\label{eqenergyequalreacdiff}
\frac{1}{2}\frac{\rd}{\rd t}|u(t)|_H^2 + a\|u(t)\|_V^2 + (f(t,u(t)),u(t)) = \langle g(t), u(t) \rangle_{V',V},
\ee
for almost every $t \in I$.
\end{enumerate}
\end{prop}


Note that, in \eqref{eqenergyequalreacdiff}, we abuse notation by denoting as $f(t,u(t))$ the mapping $x \in \Omega \mapsto f(t,u(x,t))$, which, as a consequence of \eqref{reacdiffcondf2}, belongs to $L^q(\Omega)$ for almost every $t \in I$ and for all $u \in \mU_I$.

Now we prove that the set of weak solutions $\mU_I$ satisfies the hypotheses of Theorem \ref{existencestatsol}. We first observe that Theorem \ref{existence-reacdiff} implies that $\Pi_{t_0}\mU_I = H$, so that $\mU_I$ satisfies hypothesis \eqref{eqH1b} of Theorem \ref{existencestatsol}, with $X = H$. For the remaining hypotheses of Theorem \ref{existencestatsol}, we actually prove the stronger property that \eqref{eqH2biii} holds for every compact subset of $X$. This is given in the following proposition. Hence, in this case, we can take $\fK'(X)$ to be the family of all compact subsets of $X$, so that, in particular, hypothesis \eqref{eqH2bii} is trivially true. 

\begin{prop}\label{reacdiffKUIcompact}
Let $I\subset \mathds{R}$ be an interval closed and bounded on the left with left end point $t_0$ and let $K$ be a compact subset of $H$. Then $\KUI$ is compact in $\Cloc(I,H)$.
\end{prop}
\begin{proof}
Since $\mX=\Cloc(I,H)$ is a metrizable space, it suffices to show that $\KUI$ is sequentially compact. Consider then a sequence $\{u_j\}_j$ in $\KUI$. Since $K$ is compact there exists $u_0\in K$ such that, by taking a subsequence if necessary, $u_j(t_0)\rightarrow u_0$ in $H$. This implies in particular that the sequence $\{u_j(t_0)\}$ is bounded in $H$.

Using condition \eqref{reacdiffcondf1} on the energy equality \eqref{eqenergyequalreacdiff} for each $u_j$ and integrating from $t_0$ to $t$, we obtain that
\begin{multline}
  \label{eqreacdiffest}
|u_j(t)|_H^2 - |u_j(t_0)|_H^2 + a\int_{t_0}^t \|u_j(s)\|_V^2 \;\rd s + 2\eta\int_{t_0}^t |u_j(s)|^p_{L^p} \;\rd s \\
\leq \frac{1}{a} \int_{t_0}^t \|g(s)\|_{V'}^2 \;\rd s + 2 C_1|t-t_0||\Omega|,
\end{multline}
where $|\cdot|_{L^p}$ denotes the norm in $L^p(\Omega)$, $|\Omega|$ is the Lebesgue measure of $\Omega$, and $C_1$ is as in \eqref{reacdiffcondf1}. Consider a sequence $\{J_n\}_n$ of compact subintervals of $I$ such that $I = \bigcup_n J_n$. Then, from the estimate \eqref{eqreacdiffest} and the boundedness of the sequence $\{u_j(t_0)\}$ in $H$, it follows that, for each $n$, $\{u_j\}_j$ is a bounded sequence in $L^2(J_n,V) \cap L^p(J_n,L^p(\Omega)) \cap L^\infty (J_n,H)$. Using the same arguments as in \cite[Theorem XV.3.1]{ChepVishik} and a diagonalization process, we obtain a weak solution $u$ of problem \eqref{reacdiffeq1}-\eqref{reacdiffeq2} on $I$ such that, modulo a subsequence, $u_j\rightarrow u$ in $\Cloc(I,H)$. In particular, it follows that $u_j(t)\rightarrow u(t)$ in $H$, for every $t\in I$. Thus, $u(t_0) = u_0 \in K$ and we conclude that $u\in \KUI$, as required.
\end{proof}

The existence of a trajectory statistical solution with respect to a given initial data now follows by a simple application of Theorem \ref{existencestatsol} for $X=H$ and $\mU_I$ as the set of weak solutions of \eqref{reacdiffeq1}-\eqref{reacdiffeq2} over a given interval $I \subset \mathds R$ closed and bounded on the left. Recall that since $H$ is a Polish space then every Borel probability measure on $H$ is tight. We then have the following result.

\begin{thm}
Let $I\subset\mathds{R}$ be an interval closed and bounded on the left with left end point $t_0$ and let $\mU_I$ be the set of weak solutions of problem \eqref{reacdiffeq1}-\eqref{reacdiffeq2} on $I$. If $\mu_0$ is a Borel probability measure on $H$ then there exists a $\mU_I$-trajectory statistical solution $\rho$ on $\Cloc(I,H)$ such that $\Pi_{t_0}\rho=\mu_0$.
\end{thm}

Next we obtain the existence of statistical solutions with respect to a given initial measure.

\begin{thm}
Let $I\subset\mathds{R}$ be an interval closed and bounded on the left with left end point $t_0$ and let $\mu_0$ be a Borel probability measure on $H$ satisfying
\be
\int_H |u|_H^2\;\rd \mu_0(u) < \infty.
\ee
Then there exists a projected statistical solution $\{\rho_t\}_{t\in I}$ of \eqref{reacdiffeq1}-\eqref{reacdiffeq2}, associated with a $\mU_I$-trajectory statistical solution, such that
\renewcommand{\theenumi}{\roman{enumi}}
\begin{enumerate}
  \item \label{rdemeaninitialcondition} The initial condition $\rho_{t_0}=\mu_0$ holds;
  \item \label{rdecontinuitystatinfo} The function
\be
t\mapsto \int_H \varphi(u)\;\rd\rho_t(u)
\ee
is continuous on $I$ for every bounded and continuous real-valued function $\varphi$ on $H$;
  \item For any cylindrical test function $\Phi$ in $H^{-r}(\Omega)$, with $r\geq\max\{1,n(1/2-1/p)\}$, it follows that
\begin{multline}
\int_H\Phi(u)\;\rd\rho_t(u)=\int_H\Phi(u)\;\rd\rho_{t'}(u)\\
+\int_{t'} ^t\int_H\langle a\Delta u -f(s,u) + g(s),\Phi'(u)\rangle_{H^{-r}(\Omega),H_0^r(\Omega)} \;\rd\rho_s(u)\rd s,
\end{multline}
for all $t,t'\in I$;
  \item For every nonnegative, nondecreasing and continuously-differentiable real-valued function $\psi$ with bounded derivative, the function
  \[
    t \mapsto \int_H \psi(|u|_H^2) \;\rd \rho_t(u)
  \]
  is absolutely continuous on $I$, and the following mean strengthened energy equality holds in the distribution sense on $I$,
  \[
  \frac{\rd}{\rd t}\int_H \psi(|u|_H^2) \;\rd\rho_t(u) = 2\int_H \psi'(|u|_H^2)[\langle g(t), u \rangle_{V',V} - a \|u\|_V^2 - (f(t,u),u)]\;\rd\rho_t(u).
  \]
  \end{enumerate}
\end{thm}
\begin{proof}
The proof follows by arguments similar to the ones used in Theorem \ref{thm-existence-ss-NSE}. 
We apply Theorem \ref{thm-existence-projectedss} and Proposition \ref{prop-energy-eq} with $X=H$, $\mU=\mU_I$, $Z=H_0^1(\Omega) \cap L^p(\Omega)$, $Y=H_0^r(\Omega)$, and the functions $F: I \times Z \rightarrow Y'$, $\alpha: I \times H \rightarrow \mathds{R}$ and $\beta: I \times V \rightarrow \mathds{R}$ defined as
\[F(t,u) =  a\Delta u(t) -f(t,u(t)) + g(t),\]

\[\alpha(t,u) = \psi(|u|_H^2),\]
and
\[\beta(t,u) = - 2\psi'(|u|_H^2)[\langle g(t), u \rangle_{V',V} - a \|u\|_V^2 - (f(t,u),u)].\]
\end{proof}

\subsection{Nonlinear wave equation}
\label{subsecnonlinearwaveeq}

In this section, we apply the abstract framework to prove the existence of statistical solutions of a nonlinear hyperbolic-type equation which appears within the theory of Relativistic Quantum Mechanics. We follow the framework presented in \cite[Chap. 1, Sec.1]{Lions}.

Let $\Omega\subset\mathds{R}^n$ be a bounded open set with smooth boundary, denoted by $\partial \Omega$, and let $I\subset \mathds{R}$ be an arbitrary interval. Consider the equation

\be\label{hypeq}
\frac{\partial ^2 u}{\partial t^2}-\Delta u+|u|^r u = f,
\ee
where the real-valued function $u=u(\bx,t)$ is the unknown variable, $r$ is a positive constant and $f=f(\bx,t)$ is a given function, with $\bx \in \Omega$ and $t \in I$.

We endow equation \eqref{hypeq} with the following boundary condition:
\be\label{hypboundcond}
u(\bx,t)|_{\bx\in\partial\Omega}=0, \quad \forall t\in I.
\ee

In order to obtain a functional setting for problem \eqref{hypeq}-\eqref{hypboundcond}, we introduce the space
\[
\tilde{V}=H_0^1(\Omega)\cap L^p(\Omega),
\]
where $p=r+2$. The space $\tilde{V}$ turns into a Banach space when endowed with the norm $\|\cdot\|_{\tilde V}$, defined by
\[
\|v\|_{\tilde V} = \|v\|_{H_0^1} + |v|_{L^p}, \quad \forall v\in \tilde{V},
\]
where $\|\cdot\|_{H_0^1}$ and $|\cdot|_{L^p}$ denote the usual norms in the spaces $H_0^1(\Omega)$ and $L^p(\Omega)$, respectively. 

The dual space of $\tilde{V}$ is the space $\tilde{V}'=H^{-1}(\Omega)+L^q(\Omega)$, where $1/p + 1/q = 1$. The duality product between $\tilde{V}$ and $\tilde{V}'$ is denoted by $\langle \cdot, \cdot \rangle_{\tilde{V}',\tilde{V}}$.

Also, we consider the space $L^2(\Omega)$ endowed with its usual norm and inner product, which are denoted respectively by $|\cdot|_{L^2}$ and $(\cdot,\cdot)_{L^2}$. Moreover, we assume that $f$ is a function in $L^2_{\rloc}(I,L^2(\Omega))$.

We rewrite equation \eqref{hypeq} in the following equivalent form:
\be\label{nwesystem}
\left\{\begin{array}{ll} \displaystyle \frac{\partial u}{\partial t} - v = 0, \\ \\
                         \displaystyle \frac{\partial v}{\partial t} -\Delta u+|u|^r u = f.
         \end{array}
  \right.
\ee

We denote the nonlinear term of the second equation in \eqref{nwesystem} by the function $b:\tilde{V}\rightarrow \tilde{V}'$ given by
\[
b(u)=|u|^r u, \quad \forall u \in \tilde{V}.
\]
Then, considering $U=(u,v)$ and the linear operator $A$ defined by
\[
AU = \begin{pmatrix}0 & -I \\ -\Delta & 0 \end{pmatrix}U = \begin{pmatrix}0 & -I \\ -\Delta & 0 \end{pmatrix}\begin{pmatrix}u \\ v\end{pmatrix} = \begin{pmatrix}-v \\ -\Delta u\end{pmatrix},
\]
the system \eqref{nwesystem} becomes
\be\label{nwevectoreq}
\frac{\rd U}{\rd t} + AU + N(U)= G,
\ee
where $N(U)$ and $G$ are the vectors
\[
N(U) = \begin{pmatrix}0 \\ b(u)\end{pmatrix},\quad G = \begin{pmatrix} 0 \\ f\end{pmatrix}.
\]
From \eqref{hypboundcond} we obtain the following boundary condition for \eqref{nwevectoreq}:
\be\label{hypvectorboundcond}
U(\bx,t)|_{\bx\in\partial\Omega}=0, \quad \forall t\in I.
\ee
Now we define $V = \tilde{V} \times L^2(\Omega)$, with the norm
\[
\|U\|_{V} = \sqrt{\|u\|_{\tilde{V}}^2 + |v|_{L^2}^2}, \quad \forall \, U = (u,v) \in V.
\]
When endowed with its corresponding weak topology, the space $V$ is denoted by $V_\rw$.

We characterize the dual of $V$ as the space $V' = L^2 \times \tilde{V}'$ (see Remark \ref{switcheddual}), with the duality product between $h = (f,g)\in V'$ and $U=(u,v)\in V$ as
\[
\langle h , U \rangle_{V',V} = ( f , v )_{L^2} + \langle g , u \rangle_{\tilde{V}',\tilde{V}}.
\]
With this representation, the usual norm for an element $h=(f,g)$ in the dual space $V'$ can also be written as
\[
\|h\|_{V'} = \sqrt{|f|_{L^2}^2 + \|g\|_{\tilde{V}'}^2}.
\]

We now give the definition of a weak solution of problem \eqref{nwevectoreq}-\eqref{hypvectorboundcond}.

\renewcommand{\theenumi}{\roman{enumi}}
\begin{defs}\label{hypweaksol}
Let $I\subset\mathds{R}$ be an interval and let $f\in L^2_{\rloc}(I, L^2(\Omega))$. We say that $U=U(t)=(u(t),v(t))$ is a \emph{weak solution of problem \eqref{nwevectoreq}-\eqref{hypvectorboundcond} on $I$} if the following conditions are satisfied:
\begin{enumerate}
\item\label{hypwsi} $U\in L^\infty_{\rloc}(I,V)$;
\item\label{hypwsii} $U \in \Cloc(I,V_\rw)$;
\item\label{hypwsiii} $U$ satisfies
\be\label{hypfunceq}
\frac{\rd U}{\rd t} + AU + N(U) = G \quad \text{in }V',
\ee
in the sense of distributions on $I$;
\item\label{hypwsiv} For almost every $t'\in I$, $U$ satisfies the following energy inequality
\be\label{hypenergyineq}
E(U(t)) \leq E(U(t'))+\int_{t'}^t \langle \check{G}(s),U(s) \rangle_{V',V}\;\rd s,
\ee
for every $t\in I$ with $t>t'$, where $\check{G} = (f,0)$ and
\be\label{defJ}
E(U) = E(u,v) = \frac{1}{2}\|u\|_{H_0^1}^2 + \frac{1}{p}|u|_{L^p}^p + \frac{1}{2}|v|_{L^2}^2;
\ee
\item\label{hypwsv} If $I$ is closed and bounded on the left, with left end point $t_0$, then $U$ is strongly continuous at $t_0$ from the right, i.e. $U(t)\rightarrow U(t_0)$ in $V$ as $t\rightarrow t_0^+$.
\end{enumerate}
The set of times $t'$ for which \eqref{hypenergyineq} is valid can be characterized as the points of strong continuity from the right of $U$, and they form a set of total measure in $I$.
\end{defs}

\begin{rmk}
  \label{switcheddual}
  A more natural and usual representation for the dual of $V=\tilde V \times L^2(\Omega)$ is simply $\tilde V' \times L^2(\Omega)$, preserving the order of the variables in the product space. However, in order to be able to write the equation for the evolution of the information $\int_{V}\Phi(U)\;\rd\rho_t(U)$ in the general form \eqref{meaneq} (see \eqref{meaneqnwe}), we switched the representation to $V'= L^2(\Omega) \times \tilde V$. For the same reason, we introduced the term $\check{G} = (f,0)$.
\end{rmk}

For any $R>0$, let $B_{V}(R)$ denote the closed ball of radius $R$ in $V$. We define the following sets of weak solutions of problem \eqref{nwevectoreq}-\eqref{hypvectorboundcond}:
\be
\mU_I=\{U\in \Cloc(I,V_\rw): U \mbox{ is a weak solution of \eqref{nwevectoreq}-\eqref{hypvectorboundcond} on } I\},
\ee
\be
\mU_I(R)=\{U \in  \mC(I,B_{V}(R)_\rw): U \mbox{ is a weak solution of \eqref{nwevectoreq}-\eqref{hypvectorboundcond} on } I\}.
\ee

Next we state an existence theorem of individual weak solutions for the initial value problem associated with the system \eqref{nwevectoreq}-\eqref{hypvectorboundcond}. The proof is given in \cite[Theorem 1.1, Chap. 1, Sec. 1]{Lions}. Although the regularity conditions \eqref{hypwsii}, \eqref{hypwsiv} and \eqref{hypwsv} of Definition \ref{hypweaksol} are not explicitly written in the statement of the theorem in this reference, they are obtained along the lines of their proof. 

\begin{thm}\label{hypexistence}
Let $I\subset\mathds{R}$ be an interval closed and bounded on the left with left end point $t_0$ and let $f\in L^2_{\rloc}(I,L^2(\Omega))$. Then, given $U_0\in V$, there exists at least one weak solution $U\in \mU_I$ of \eqref{nwevectoreq}-\eqref{hypvectorboundcond}, in the sense of Definition \ref{hypweaksol}, satisfying $\Pi_{t_0}U=U_0$.
\end{thm}

Consider now $I \subset \mathds R$ an interval closed and bounded on the left with left end point $t_0$. In this case, item \eqref{hypwsv} of Definition \ref{hypweaksol} implies that the energy inequality \eqref{hypenergyineq} is valid for $t' = t_0$.

Let $U = (u,v) \in \mU_I$ such that $U(t_0) \in B_V(R)$ , for some $R>0$. From the energy inequality \eqref{hypenergyineq} with $t' = t_0$, it follows that
\be\label{eqestnormU}
E(U(t)) \leq R + \frac{1}{2}\int_{t_0}^t |f(s)|_{L^2}^2 \;\rd s + \frac{1}{2}\int_{t_0}^t |v(s)|_{L^2}^2 \;\rd s,
\ee
for every $t \in I$, which also yields
\be\label{eqestnormv}
\frac{1}{2}|v(t)|_{L^2}^2 \leq R + \frac{1}{2}\int_{t_0}^t |f(s)|_{L^2}^2 \;\rd s + \frac{1}{2}\int_{t_0}^t |v(s)|_{L^2}^2 \;\rd s.
\ee
Then, given a compact subinterval $J \subset I$, by applying Gr\"onwall's inequality in \eqref{eqestnormv} we obtain that $|v(\cdot)|_{L^2}^2$ is uniformly bounded on $J$. From the estimate \eqref{eqestnormU}, it then follows that there exists $\tilde{R} \geq R$ such that $U(t) \in B_V(\tilde{R})$, for every $t \in J$. Thus, the restriction of $U$ to $J$ belongs to $\mU_J(\tilde{R})$.

We shall now prove that the set of weak solutions $\mU_I$ satisfies the hypotheses \eqref{eqH1b}, \eqref{eqH2bii}, and \eqref{eqH2biii} of Theorem \ref{existencestatsol}.

Theorem \ref{hypexistence} shows, in an equivalent form, that $\Pi_{t_0}\mU_I=V_\rw$. Thus, the set $\mU_I$ satisfies hypothesis \eqref{eqH1b} with $X=V_\rw$.

Now define
\[\fK'(V_\rw)=\{K\subset V_\rw\,|\, K\text{ is a (strongly) compact set in }V\}.
\]

Since $V$ is a separable Banach space, we obtain that every Borel probability measure $\mu_0$ on $V_\rw$ is tight with respect to the family $\fK'(V_\rw)$ (see Sections \ref{subsecelementsofmeastheory} and \ref{subsecinjectborel}). Then, considering $X=V_\rw$, it follows that the family $\fK'(V_\rw)$ satisfies hypotheses \eqref{eqH2bii} of Theorem \ref{existencestatsol}. The next proposition shows that $\mU_I$ also satisfies hypothesis \eqref{eqH2biii}.

\begin{prop}\label{hypKUIcompact}
Let $I\subset\mathds{R}$ be an interval closed and bounded on the left with left end point $t_0$ and let $K$ be a set in $\fK'(V_\rw)$. Then $\KUI$ is a compact set in $\mX=\Cloc(I,V_\rw)$.
\end{prop}
\begin{proof}
Let $R>0$ be such that $K\subset B_{V}(R)$ and let $\{J_n\}_n$ be a sequence of compact subsets of $I$ such that
\[
I=\bigcup_n J_n.
\]
Then, from the energy inequality \eqref{hypenergyineq} with $t'=t_0$, one obtains that, for every $n\in \mathds{N}$, there exists a positive real number $R_n\geq R$ such that
\[
\Pi_{J_n} U \in \mU_{J_n}(R_n),\quad \forall U \in \KUI,
\]
which implies that
\[
\KUI \subset \bigcap_n \Pi_{J_n}^{-1}\mU_{J_n}(R_n).
\]
Since $\mU_{J_n}(R_n)$ is a subset of $\Cloc(I,B_{V}(R_n)_\rw)$, which is a metrizable space, it follows that $\KUI$ is also metrizable. Thus, it is enough to prove that $\KUI$ is a sequentially compact space.

Consider then a sequence $\{U_k\}_k$ in $\KUI$. Since $U_k(t_0)\in K$ and $K$ is a compact set in $V$, there exists $U_0\in V$ and a subsequence $\{k_j\}_j$ such that
\be\label{hyppropeq6}
U_{k_j}(t_0) \rightarrow U_0 \quad \text{in }V.
\ee

Following classical arguments used for the existence of weak solutions (see \cite[Chap. 1, Sec.1]{Lions}), we obtain a~priori estimates that allow us to pass to the limit on each compact set $J_n$. Then, using a diagonalization process, we obtain a further subsequence (which we still denote by $\{U_{k_j}\}_j$) and a function $U$ defined on the interval $I$ such that $\{U_{k_j}\}_j$ converges to $U$ in $\Cloc(I,V_\rw)$ and $U$ is a weak solution on the interior of $I$. Thanks to \eqref{hyppropeq6} we have at the initial time that $U(t_0) = U_0 \in K$, so that $U\in \Pi_{t_0}^{-1}K$. Then, as in the case of the Navier-Stokes equations (see the proof of Proposition \ref{prop-KUI-compact}), using the energy inequality and the fact that the convergence \eqref{hyppropeq6} at the initial time is in the strong topology, we obtain that $U$ is strongly continuous at the initial time $t_0$, so that $U\in \mU_I$. Therefore, $U\in \KUI$, proving that $\KUI$ is compact.
\end{proof}

Thus, applying Theorem \ref{existencestatsol} with $X = V_\rw$ and $\mU$ as the set of weak solutions $\mU_I$, we obtain the following result on the existence of a trajectory statistical solution for the nonlinear wave equation with respect to a given initial data.

\begin{thm}
Let $I\subset\mathds{R}$ be an interval closed and bounded on the left with left end point $t_0$ and let $\mU_I$ be the set of weak solutions of problem \eqref{nwevectoreq}-\eqref{hypvectorboundcond} on $I$. If $\mu_0$ is a Borel probability measure on $V$ then there exists a $\mU_I$-trajectory statistical solution $\rho$ on $\Cloc(I,V_\rw)$ such that $\Pi_{t_0}\rho=\mu_0$.
\end{thm}

In the next result we prove the existence of a statistical solution of problem \eqref{nwevectoreq}-\eqref{hypvectorboundcond} in the sense of Definition \ref{def-statsolphasespace} with respect to a given initial data.

\begin{thm}
Let $I\subset\mathds{R}$ be an interval closed and bounded on the left with left end point $t_0$ and let $\mU_I$ be the set of weak solutions of problem \eqref{nwevectoreq}-\eqref{hypvectorboundcond} on $I$. Consider a Borel probability measure $\mu_0$ on $V$ satisfying
\be\label{eqfiniteenergynwe}
\int_{V} E(U)\;\rd \mu_0(U) < \infty,
\ee
with $E$ as defined in \eqref{defJ}.
Then there exists a projected statistical solution $\{\rho_t\}_{t\in I}$ of \eqref{hypeq}, associated with a $\mU_I$-trajectory statistical solution, such that
\renewcommand{\theenumi}{\roman{enumi}}
\begin{enumerate}
  \item \label{meaninitialconditionnwe} The initial condition $\rho_{t_0}=\mu_0$ holds;
  \item \label{continuitystatinfonwe} The function
\be
t\mapsto \int_{V} \varphi(U)\;\rd\rho_t(U)
\ee
is continuous on $I$, for every bounded and weakly-continuous real-valued function $\varphi$ on $V$, and is measurable on $I$, for every bounded and continuous real-valued function $\varphi$ on $V$;
  \item \label{meaneqnwecondition} For any cylindrical test function $\Phi$ in $V'$, it follows that
\begin{multline}
  \label{meaneqnwe}
\int_{V}\Phi(U)\;\rd\rho_t(U)=\\ \int_{V}\Phi(U)\;\rd\rho_{t'}(U)+\int_{t'} ^t\int_{V}\langle G - AU - N(U) ,\Phi'(U)\rangle_{V',V}\;\rd\rho_s(U)\rd s,
\end{multline}
for all $t,t'\in I$;
  \item\label{existssivnwe} The mean strengthened energy inequality
  \be\label{meanstrengthenedenergyineqnwe}
  \frac{\rd}{\rd t}\int_{V} \psi(E(U)) \;\rd\rho_t(U) \leq \int_{V} \psi'(E(U))\langle \check{G}(t),U(t)\rangle_{V',V} \;\rd\rho_t(U)
  \ee
  is satisfied in the distribution sense on $I$, for every nonnegative, nondecreasing and continuously-differentiable real-valued function $\psi$ with bounded derivative, where $\check{G} = (f,0)$;
  \item \label{meancontoriginnwe} At the initial time, the limit
  \be\label{eqmeancontoriginnwe} \lim_{t\rightarrow t_0^+} \int_{V} \psi(E(U)) \;\rd\rho_t(U) = \int_{V} \psi(E(U)) \;\rd\mu_0(U)
  \ee
  holds for every function $\psi$ as in \eqref{existssivnwe}.
\end{enumerate}
\end{thm}
\begin{proof}
The proof follows by arguments similar to the ones used in Theorem \ref{thm-existence-ss-NSE}, considering $X=V_\rw$, $Z=Y=V$ and the functions $F:I\times V\rightarrow V'$, $\alpha:I\times V_\rw \rightarrow \mathds{R}$ and $\beta:I\times V \rightarrow\mathds{R}$ defined respectively as
\[F(t,U)=G(t)-AU(t)-N(U(t)),
\]
\[\alpha(t,U)=\psi(E(U(t))),
\]
and
\[\beta(t,U)= -\psi'(E(U(t)))\langle \check{G}(t),U(t)\rangle_{V',V},
\]
for every $(t,U)$ in the corresponding domains.
\end{proof}

\section*{Acknowledgments}
The authors would like to thank Professors Dinam\'erico Pombo, for discussions about general topology, Nilson Bernardes Jr., for discussions related to the results considered in Section \ref{subsecinjectborel}, Dario Darji for discussions pertaining to the Pettis integral, and Fabio Ramos, for bringing to our attention the works of Topsoe. The last author, R. Rosa, is also greatly indebted to Professors Roger Temam and Ciprian Foias, for their continued mentoring and support and for being inspirations for his work. The three authors also acknowledge the financial support of CNPq-Brazil.

\end{document}